\documentclass{amsart}
\usepackage{amssymb}
\usepackage{amsfonts}
\usepackage[alphabetic,initials]{amsrefs}
\usepackage{graphicx}
\usepackage{eucal}
\usepackage{enumerate}

\newcommand{\B}{\mathcal{B}}
\newcommand{\R}{\mathbb{R}}

\renewcommand{\H}{\mathcal{H}}
\renewcommand{\L}{\mathcal{L}}

\newcommand{\F}{\mathcal{F}}
\newcommand{\loc}{\mathrm{loc}}
\newcommand{\step}[1]{\par\medskip\par\noindent\textit{#1}}

\newcommand{\LL}{\lfloor}

\newcommand{\supp}{\operatorname{supp}}

\newcommand{\dist}{\operatorname{dist}}
\renewcommand{\div}{\operatorname{div}}

\newcommand{\avint}{-\kern-10.7pt\int}

\def\XXint#1#2#3{{\setbox0=\hbox{$#1{#2#3}{\int}$}
\vcenter{\hbox{$#2#3$}}\kern-0.5\wd0}}

\mathchardef\ordinarycolon\mathcode`\:
\mathcode`\:=\string"8000
\begingroup \catcode`\:=\active
  \gdef:{\mathrel{\mathop\ordinarycolon}}
\endgroup

\renewcommand{\phi}{\varphi}
\renewcommand{\epsilon}{\varepsilon}

\theoremstyle{plain}
\newtheorem{theorem}{Theorem}[section]
\newtheorem{lemma}[theorem]{Lemma}
\newtheorem{proposition}[theorem]{Proposition}
\newtheorem{corollary}[theorem]{Corollary}
\theoremstyle{definition}

\theoremstyle{remark}
\newtheorem{remark}[theorem]{Remark}
\newtheorem*{claim}{Claim}
\numberwithin{equation}{section}

\title{Singular Perturbation Problem in Boundary/Fractional Combustion}

\author{Arshak Petrosyan}
\address{Department of Mathematics, Purdue University, West Lafayette,
  IN 47907, USA}
\email{arshak@math.purdue.edu}
\thanks{A.P. was supported in part by NSF grant DMS-1101139}

\author{Wenhui Shi}
\address{Mathematics Institute, Universit\"at Bonn, Endenicher Allee 62, 53115 Bonn, Germany}
\email{wenhui.shi@hcm.uni-bonn.de}
\thanks{W.S. was supported in part by NSF grant DMS-1101139 and the Hausdorff Center of Mathematics}

\author{Yannick Sire}
\address{Universit\'e Aix-Marseille, Institut de Math\'ematiques de Marseille, CMI, Technopole de Ch\^ateau-Gombert, Marseille,  France}
\email{yannick.sire@univ-amu.fr}
\thanks{Y.S. was supported in part by ANR projects HAB and NONLOCAL}


\begin{document}

\begin{abstract}
  Motivated by a nonlocal free boundary problem, we study uniform
  properties of solutions to a singular perturbation problem for a
  boundary-reaction-diffusion equation, where the reaction term is of
  combustion type. This boundary problem is related to the fractional
  Laplacian. After an optimal uniform H\"older regularity is shown, we
  pass to the limit to study the free boundary problem it leads to.
\end{abstract}

\maketitle

\tableofcontents
\section{Introduction}

In this paper we study nonnegative solutions for the semilinear
boundary-reaction-diffusion problem:
\begin{equation}\label{eq:P_eps}
  \tag{$P_\epsilon$}
  \begin{aligned}
    \L_s u_\epsilon:=\div(|x_n|^{1-2s} \nabla u_\epsilon)= 0&\quad\text{in } B_1^+=B_1\cap\{x_n>0\},\\
    -\lim_{x_n\to0+}x_n^{1-2s}\frac{\partial u_\epsilon}{\partial
      x_n}=-\beta _\epsilon(u_\epsilon)&\quad\text{on } B_1'=B_1\cap
    \{x_n=0\},
  \end{aligned}
\end{equation}
where $B_1$ is the unit ball in $\R^n$, $n\geq 2$, $s\in (0,1)$, and
$\epsilon$ is a small positive parameter. The nonlinear reaction term
$\beta _\epsilon(t)$ is of combustion type and is given by
\begin{equation}\label{eq:beta-eps}
  \beta_\epsilon(t)=\frac{1}{\epsilon}
  \beta\left(\frac{t}{\epsilon}\right),\quad t\in\R,
\end{equation}
with $\beta \in C^{0,1}_c(\R)$ satisfying
\begin{equation}\label{eq:beta}
  \beta\geq 0,\quad \supp\beta=[0,1],\quad\text{and}\quad \int_0^1\beta(t)dt =M.
\end{equation}
Note that the solutions of \eqref{eq:P_eps} are the critical points
(including the local minimizers) of the energy functional
\begin{equation}\label{eq:energy-eps}
  J_\epsilon(u)=\int_{B_1^+}|\nabla u|^2 |x_n|^{1-2s}+\int_{B_1'}2\B_\epsilon(u)
\end{equation}
among all functions in the weighed Sobolev space
$W^{1,2}(B_1^+, |x_n|^{1-2s})$ with fixed trace on
$(\partial B_1)^+=\partial B_1\cap\{x_n>0\}$, where $\B_\epsilon$ is
the primitive of $\beta_\epsilon$ given by
$$
\B_\epsilon(t)=\int_{0}^t\beta_\epsilon(s)ds.
$$
Formally, as $\epsilon\to 0+$, the functional $J_\epsilon$ converges
to
$$
J_0(u)=\int_{B_1^+}|\nabla u|^2
|x_n|^{1-2s}+\int_{B_1'}2M\chi_{\{u>0\}},
$$
which is the boundary (or thin) analogue of the Alt-Caffarelli
\cite{AC} energy functional. The study of the minimizers of $J_0$ has
been initiated in \cite{CRS} and by now there is a good understanding
of the associated free boundary problem. Namely, it is known that the
minimizers of $J_0$ solve (in the appropriate sense)
\begin{equation}\label{eq:freebdry}\tag{$P$}
  \begin{aligned}
    \div (|x_n|^{1-2s} \nabla u)=0&\quad\text{in }B_1^+,\\
    -\lim_{x_n\to0+}x_n^{1-2s}\frac{\partial u}{\partial x_n}=0&\quad\text{on } \{u>0\} \cap B_1',\\
    \lim_{t\to0+}\frac{u(x_0+t\nu'_{x_0})}{t^s}=
    \sqrt{\frac{2M}{c_0(s)}}&\quad\text{for }x_0\in \F_u,
  \end{aligned}
\end{equation}
where
$$
\F_u:=\partial\{u(\cdot,0)>0\}\cap B_1'
$$
is the \emph{free boundary} in the problem, $\nu'_{x_0}$ is the
in-plane, inner unit normal to $\{u(\cdot,0)>0\}$ and and $c_0(s)>0$
is a constant. The regularity properties of the free boundary $\F_u$
for the minimizers in the case $s=1/2$ have been studied in the series
of papers \cites{DR,DS1,DS2}, establishing the smoothness of flat free
boundaries. For the general $s\in(0,1)$, the $C^{1,\alpha}$ regularity
of flat free boundaries has been established in \cite{DSS}.

One of our main objectives in this paper is to show that the solutions
$u_\epsilon$ of the singular perturbation problem \eqref{eq:P_eps},
also converge to a solution to the free boundary problem
\eqref{eq:freebdry}, in a certain, weaker, sense. We show the uniform
$s$-H\"older regularity of $u_\epsilon$
(Theorem~\ref{thm:uniform_holder}), however, the passage to the limit
$u$ as $\epsilon\to 0+$ is complicated by the fact that
$\B_\epsilon(u_\epsilon)$ may not converge (in weakly-$*$ sense) to
$M\chi_{\{u>0\}}$. Nevertheless, at free boundary points $x_0\in\F_u$
with a measure-theoretical normal and a nondegeneracy condition on
$u$, we can establish an asymptotic development of $u$, implying the
free boundary condition in \eqref{eq:freebdry}
(Theorem~\ref{thm:maintheorem}).

This kind of convergence results are very well known in combustion
theory for the singular perturbation problems of the type
$$
\Delta u_\epsilon=\beta_\epsilon(u_\epsilon)\quad\text{in }B_1,
$$
(even in time-dependent case) with $\beta_\epsilon$ as in
\eqref{eq:beta-eps}, since the works of Zel'dovich and
Frank-Kamenetski\u\i\ \cite{ZFK}. Mathematically rigorous results,
however, are much more recent. Here we cite some of the important ones
for our paper: \cites{BCN,CV,Vazquez-zako,CLW2,CLW,DPS,Weiss}.

The singular-perturbation problem \eqref{eq:P_eps} can be also viewed
as the localized version of the global reaction-diffusion equation
\begin{equation}\label{eq:P_eps'}\tag{$P_\epsilon'$}
  \begin{aligned}
    (-\Delta_{x'})^s u_\epsilon
    =-\beta_\epsilon(u_\epsilon)&\quad\text{in
    }\Omega\subset\R^{n-1}\\
    u_\epsilon=g_\epsilon&\quad\text{on }\R^{n-1}\setminus\Omega
  \end{aligned}
\end{equation}
for the fractional Laplacian $(-\Delta_{x'})^s$ in
$x'=(x_1,\ldots,x_{n-1})$ variables, where $g_\epsilon$ is a
nonnegative function on $\R^{n-1}\setminus\Omega$ having the meaning
of the boundary data. We recall that the fractional Laplacian is
defined as the Fourier multiplier of symbol $|\xi'|^{2s}$ for
$s \in (0,1)$ (see \cite{landkof} for a treatment of these operators).
Note that the solutions of \eqref{eq:P_eps'} are the critical points
of the energy functional
$$
j(v)=c_{n,s}\int_{\R^{n-1}}\int_{\R^{n-1}}\frac{(u(x')-u(y'))^2}{|x'-y'|^{n-1+2s}}+\int_{\R^{n-1}}
2\B_\epsilon(u),
$$
among all functions such that $u=g_\epsilon$ on
$\R^{n-1}\setminus\Omega$.  (Here $c_{n,s}>0$ is a normalization
constant.)

The connection between \eqref{eq:P_eps} and \eqref{eq:P_eps'} is then
established through the so-called Caffarelli-Silvestre extension
\cite{caffaSil}: if for a given function $u$ on $\R^{n-1}$ (with
appropriate growth conditions at infinity) we consider the extension
$\tilde{u}$ to $\R^n_+=\R^{n-1}\times(0,\infty)$ by solving the
Dirichlet problem
\begin{align*}
  \L_s\tilde{u}=\div(x_n^{1-2s} \nabla \tilde{u})  = 0&\quad\text{in } \R^n_+\\
  \tilde{u}=u&\quad  \text{on } \R^{n-1}\times \{0\}
               \intertext{then}
               -c_{n,s}\lim_{x_n\to0+}x_n^{1-2s}\frac{\partial \tilde{u}}{\partial x_n} = (-\Delta_{x'})^s u&\quad \text{on } \R^{n-1}\times \{0\}
\end{align*}
for a positive constant $c_{n,s}$. Hence, if $u_\epsilon$ solves
\eqref{eq:P_eps'}, $x_0\in \Omega$ and $R>0$ are such that
$B'_R(x_0)\subset\Omega$, then the extension of $u_\epsilon$ to
$\R^{n}_+$ constructed as above will solve \eqref{eq:P_eps} in
$B_R^+(x_0)$. As a consequence, the singular perturbation problem
\eqref{eq:P_eps'} for the fractional Laplacian, becomes a boundary (or
thin) singular perturbation problem \eqref{eq:P_eps} for the operator
$\L_s$ in one dimension higher.

\subsection*{Main results and the structure of the paper}

In this paper, we will focus on the uniform estimate of the solutions
to $(P_\epsilon)$ and the proof of the free boundary condition in
\eqref{eq:freebdry}.

\begin{enumerate}[$\bullet$]
\item In \S\ref{s:holder} we prove the uniform $s$-H\"older regularity
  for the solutions of \eqref{eq:P_eps}, see
  Theorem~\ref{thm:uniform_holder}. This allows to pass to the limit
  as $\epsilon\to 0+$ and study the resulting solutions in the
  subsequent sections.

\item In \S\ref{s:limit}, we prove various results concerning the
  limits of $u_\epsilon$, or, more precisely, the limits of the pairs
  $(u_\epsilon,\B_\epsilon(u_\epsilon))$, which we denote
  $(u,\chi)$. The results include the compactness lemma
  (Lemma~\ref{lem:compactnesslemma}), ensuring the convergence in the
  proper spaces and Weiss-type monotonicity formulas for $u_\epsilon$
  and $(u,\chi)$ (Theorems~\ref{thm:monotonicity} and
  \ref{thm:monotonicity2}).
\item In \S\ref{s:asymptotic} we prove that the free boundary
  condition in problem \eqref{eq:freebdry} is satisfied at free
  boundary points with measure-theoretical normal for
  $\{u(\cdot,0)>0\}$, under the additional nondegeneracy condition
  (Theorem~\ref{thm:maintheorem}). This is done by identifying the
  blowups with flat free boundaries (Proposition~\ref{prop:basicex}).
  We conclude the paper by proving two additional propositions related
  to the Weiss energy at nondegenerate points
  (Propositions~\ref{prop:classify} and \ref{prop:classify2}).
\end{enumerate}

\subsection*{Notations and preliminaries}
\begin{enumerate}[$\bullet$]

\item We will use fairly standard notations in this paper.
  \begin{enumerate}[$\circ$]
  \item $\R^n$ will stand for the $n$-dimensional Euclidean space;
  \item For every $x\in\R^{n}$ we write $x=(x',x_n)$, where
    $x'=(x_1,\ldots,x_{n-1})\in\R^{n-1}$. This identifies $\R^n$ with
    $\R^{n-1}\times\R$. We also don't distinguish between $(x',0)$ and
    $x'$, thus identifying $\R^{n-1}$ with
    $\R^{n-1}\times\{0\}\subset\R^n$.
  \item $\R^n_{\pm}=\R^n\cap\{\pm x_n>0\}$;
  \item Balls and half-balls: $B_r(x)=\{y\in \R^n: |y-x|<r\}$,
    $B_r^\pm(x)=B_r(x)\cap \{\pm x_n>0\}$,
  \item `Thin' balls: $B'_r(x)=B_r(x)\cap \{x_n=0\}$.
  \item Typically, we skip the center in the notation for balls if it
    is the origin. Thus, $B_1=B_1(0)$, $B_1'=B_1'(0)$, etc.
  \end{enumerate}

\item For the functions $\beta$ and $\beta_\epsilon$, we make the
  following assumption throughout the paper. Besides
  \eqref{eq:beta-eps}--\eqref{eq:beta}, we fix a constant $A>0$ such
  that
$$
\max\{ |\beta(s)|, |\beta'(s)|\}\leq A, \quad \text{for all } s\in\R.
$$
We will also need to make a technical assumption that
$$
\beta>0\quad\text{on }(0,1).
$$
\item The functions $\B, \B_\epsilon :\R\rightarrow \R$ are the
  primitives of $\beta$ and $\beta_\epsilon$ given by
  \begin{align}\label{eq:Beps}
    \B(t)=\int_0^t \beta(s)ds,\quad
    \B_\epsilon(t)=\int_0^t\beta_\epsilon(s)ds=\B(s/\epsilon)
  \end{align}

\item \emph{Even extension of $u_\epsilon$ and weak solutions of
    \eqref{eq:P_eps}}.  In what follows, we will be extending the
  functions $u_\epsilon$ in $B_1^+$ with even reflection to all of
  $B_1$:
$$
u_\epsilon(x',-x_n)=u_\epsilon(x',x_n),\quad\text{for }x\in B_1^+.
$$
With such an extension in mind, $u_\epsilon$ is a \emph{weak solution}
of \eqref{eq:P_eps} if for any test function $\phi\in C^\infty_c(B_1)$
\begin{equation}\label{eq:variational}
  \int_{B_1}|x_n|^{1-2s}\nabla
  u_\epsilon \cdot \nabla \phi dx+
  \int_{B_1'}2\beta_\epsilon (u_\epsilon)\phi dx'
  =0,
\end{equation}
or, in other words,
\begin{equation*}
  \div( |x_n|^{1-2s}\nabla u_\epsilon)  = 2\beta_\epsilon(u_\epsilon)\chi_{\{x_n=0\}}
\end{equation*}
in the sense of distributions.

\emph{Unless specified otherwise, by a solution of \eqref{eq:P_eps} we
  will always understand a weak solution of \eqref{eq:P_eps}.}

We also note that for functions which are even symmetric in $x_n$, the
energy functional \eqref{eq:energy-eps} can be rewritten as
$$
J_\epsilon(v)=\frac12\int_{B_1}|x_n|^{1-2s}|\nabla
v|^2+\int_{B_1'}2\B_\epsilon(v).
$$
\item \emph{Rescalings.}  Finally, throughout the paper we will make
  the extensive use of rescalings. For a given $x_0\in B_1'$ and
  $\lambda>0$ define
$$
u_{\epsilon,\lambda}(x)=u_{\epsilon,\lambda}^{x_0}(x):=\frac{u(x_0+\lambda
  x)}{\lambda^s},\quad x\in B_{(1-|x_0|)/\lambda}.
$$
A straightforward computation shows that $u_{\epsilon,\lambda}$
satisfies
$$
\div(|x_n|^{1-2s}\nabla
u_{\epsilon,\lambda})=2\beta_{\epsilon/\lambda^s}(u_{\epsilon,\lambda})\H^{n-1}\LL\{x_n=0\}\quad\text{in
} B_{(1-|x_0|)/\lambda}.
$$
\end{enumerate}

\section{Uniform $C^{0,s}$ regularity} \label{s:holder}

In this section we prove the following uniform H\"older regularity
result for the solutions of \eqref{eq:P_eps}.

\begin{theorem}[Uniform $s$-H\"older estimate]
  \label{thm:uniform_holder}
  Let $u_\epsilon$ be a nonnegative solution of \eqref{eq:P_eps} with
  $\|u_\epsilon\|_{L^\infty(B_1)}\leq L$. Then
  $u_\epsilon \in C^{0,s}(K)$ for any $K\Subset B_1$ with
$$
\|u_\epsilon\|_{C^{0,s}(K)}\leq C(n,s,A,L,K)
$$
uniformly for all $\epsilon\in (0,1)$.
\end{theorem}

Our proof follows the ideas from \cite{DPS} in the case of
$p$-harmonic functions. One of the main steps is the following
Harnack-type inequality.

\begin{lemma}[Harnack-type inequality]\label{lem:harnack}
  Let $v$ be a locally bounded nonnegative weak solution of
$$0\leq \div(|x_n|^{1-2s} \nabla v) \leq A\chi_{\{0<v<1\}}\mathcal
H^{n-1}\LL\{x_n=0\}\quad\text{in } B_1$$
with $v(0)\leq 1$. Then there exists a constant $C=C(n,s,A)$ such that
$$ \|v\|_{L^\infty(B_{1/4})}\leq C.$$
\end{lemma}

To prove this lemma, we will need the following interior H\"older
estimate.

\begin{lemma}[Interior $s$-H\"older
  estimate]\label{lem:frac-int-est-holder} Let $|w|\leq M$ be a weak
  solution of
$$
|\div(|x_n|^{1-2s}\nabla w)|\leq
\mu\,\H^{n-1}\LL\{x_n=0\}\quad\text{in }B_1(x_0).
$$
Here $x_0$ is not necessarily on $\{x_n=0\}$.

Then $w\in C^{0,s}(B_{1/2}(x_0))$ with
$$
\|w\|_{C^{0,s}(B_{1/2}(x_0))}\leq C(n,s,\mu,M).
$$
\end{lemma}
\begin{proof}\mbox{}
  \begin{enumerate}[(i)]
  \item When $x_0=0$, or more generally, $(x_0)_n=0$, we refer to
    Remark~5.2 and the proof of Theorem~5.1 in \cite{ALP}.

  \item The case of general $x_0$ is obtained by considering the
    subcases
    \begin{enumerate}[(a)]
    \item $|(x_0)_n|>3/4$, and
    \item $|(x_0)_n|\leq 3/4$.
    \end{enumerate}
    In the subcase (a) the equation is uniformly elliptic in
    $B_{5/8}(x_0)$, and the estimate follows from standard interior
    estimates for uniformly elliptic equations. In the subcase (b),
    the $s$-H\"older continuity in $B_{1/2}(x_0)\cap\{|x_n|\leq 1/8\}$
    is obtained from the case (i) above. The $s$-H\"older continuity
    in $B_{1/2}(x_0)\cap\{|x_n>1/8\}$ is obtained from the uniform
    ellipticity of the operator $\L_s$ in
    $B_1(x_0)\cap\{|x_n|>1/8\}$.\qedhere
  \end{enumerate}
\end{proof}

\begin{proof}[Proof of Lemma~\ref{lem:harnack}]
  We start by an observation that the function $v$ is continuous by
  Lemma~\ref{lem:frac-int-est-holder}. We will use this fact
  implicitly throughout the proof.

  We argue by contradiction. Assuming that the conclusion of the lemma
  fails, there exists a sequence of nonnegative solutions $v_k$ with
$$
v_k(0)\leq 1\quad\text{but}\quad \|v_k\|_{L^\infty(B_{1/4})}\geq k.
$$
Let $\Omega_k:=\{x\in B_1': v_k(x)\leq 1\}$,
$O_k:=\{x\in B_1: \text{dist}(x,\Omega_k)\leq \frac{1}{3}(1-|x|)\}$,
and
$$m_k:=\max_{O_k}(1-|x|)v_k(x).$$
Observe that $B_{1/4}\subset O_k$. Then
$m_k\geq \frac{3}{4}\sup_{B_{1/4}}v_k\geq \frac{3}{4}k$. Let
$x_k\in O_k$ such that $(1-|x_k|)v_k(x_k)=m_k$, then
\begin{equation}\label{eq:lower_bound_v}
  v_k(x_k)\geq m_k\geq \frac{3}{4}k.
\end{equation}
Consider the distance $\delta_k:=\text{dist}(x_k,\Omega_k)$ and
$y_k\in \Omega_k$ realize $\delta_k$. By \eqref{eq:lower_bound_v},
$\delta_k>0$. Using the fact that $\delta_k\leq \frac{1}{3}(1-|x_k|)$
and the triangle inequality, we obtain that
$B_{\delta_k/2}(y_k)\subset O_k$ and for any
$z\in B_{\delta_k/2}(y_k)$,
\begin{equation}\label{eq:lower_bound}
  v_k(z)\leq \frac{m_k}{1-|z|}=\frac{1-|x_k|}{1-|z|}v_k(x_k)\leq 2v_k(x_k).
\end{equation}
Since $v_k$ satisfies the homogeneous equation
$\div(|x_n|^{1-2s} \nabla v_k)=0$ in $B_{\delta_k}(x_k)$, by the
Harnack inequality in \cite{FKS} there exists $c=c(n,s)$ such that
\begin{align*}
  \inf_{B_{3\delta_k/4}(x_k)} v_k \geq c\,
  v_k(x_k).
\end{align*}
In particular, since
$\overline{ B_{\delta_k/4}(y_k)}\cap \overline{
  B_{3\delta_k/4}(x_k)}\neq \emptyset$, then
\begin{equation}\label{eq:upper_bound}
  \sup_{B_{\delta_k/4}(y_k)}v_k \geq cv_k(x_k).
\end{equation}
Define
$$ w_k(x):= \frac{v_k(y_k+\delta_k x)}{v_k(x_k)}, \quad x\in B_{1/2}.$$
From \eqref{eq:lower_bound} and \eqref{eq:upper_bound} we have
\begin{equation*}
  \sup_{B_{1/2}}w_k\leq 2\quad \text{and }  \sup_{B_{1/4}}w_k\geq c.
\end{equation*}
Moreover, using \eqref{eq:lower_bound_v} and recalling that
$v_k(y_k)\leq 1$, $w_k$ satisfies
\begin{gather*}
  0\leq \div(|x_n|^{1-2s} \nabla w_k) \leq
  \frac{4A\delta_k^{2-2s}}{3k}\mathcal
  H^{n-1}\LL\{x_n=0\}  \quad\text{in } B_{1/2}\\
  w_k\geq 0,\quad w_k(0)\leq \frac{4}{3k}.
\end{gather*}
Now, invoking Lemma~\ref{lem:frac-int-est-holder}, we obtain that
$w_k$ are uniformly $s$-H\"older continuous on compact subsets of
$B_{1/2}$ and hence, over a subsequence, they will converge locally
uniformly to a function $w_0$ which satisfies
\begin{align*}
  \div(|x_n|^{1-2s} \nabla w_0) =0\quad\text{in } B_{1/2},\quad 
  \sup _{B_{1/4}}w_0\geq c>0,\quad
  w_0\geq 0,\quad  w_0(0)=0
\end{align*}
This is a contradiction to the strong maximum principle in \cite{FKS}.
\end{proof}

Now we prove the uniform $C^{0,s}$ regularity of $u_\epsilon$.

\begin{proof}[Proof of Theorem~\ref{thm:uniform_holder}] Note that it
  will be sufficient to prove the uniform estimates for small
  $0<\epsilon<\epsilon_0$, with universal $\epsilon_0$, as the
  estimate for $\epsilon_0<\epsilon<1$ will follow from
  Lemma~\ref{lem:frac-int-est-holder}. It will also be sufficient to
  give the proof for $K=B_{1/8}$. Throughout the proof, we let
  $$\Omega_\epsilon:=\{x\in B_1':u_\epsilon\leq \epsilon\}.$$

  \step{Step 1}. We will show that there exists a constant
  $C=C(n,s,A)$ such that
  \begin{align*}
    u_\epsilon(x)\leq \epsilon + C\text{dist}(x,\Omega_\epsilon)^{s}, \quad x\in B'_{1/4}\setminus\Omega_\epsilon.
  \end{align*}
  The proof is based on the construction of a proper lower barrier function.\\
  Given $x_0\in B'_{1/4}\setminus \Omega_\epsilon$, let
$$m_0:= u_\epsilon(x_0)-\epsilon, \quad \delta_0:=\text{dist}(x_0,\Omega_\epsilon).$$
We are going to show that
$$m_0\leq C(n,s,A)\delta_0^{s}.$$
By the Harnack inequality (see \cite{CS}), there exists a constant
$c_{n,s}$ such that
$$u_\epsilon(x)-\epsilon\geq c_{n,s} m_0, \quad \text{for any }x\in B_{\delta_0/2}(x_0).$$
Next, we construct an auxiliary function as follows: Let
$A_{1/2,2}:=B_2\setminus \overline{B_{1/2}}$,
$A'_{1,2}:=B_2'\setminus B_1'$, and $D:=A_{1/2,2}\setminus
A'_{1,2}$.
Let $\phi:D\rightarrow \R$ be the solution to the following Dirichlet
problem
\begin{align*}
  \div(|x_n|^{1-2s} \nabla \phi)=0&\quad\text{in }D, \\
  \phi=1&\quad\text{on } \partial B_{1/2},\\
  \phi=0&\quad \text{on } \partial B_2\cup A'_{1,2}.
\end{align*}
By the boundary Hopf lemma and boundary growth estimate (see
\cite{CS}) as well as the symmetry of $\phi$, the function $\phi$ has
the following asymptotics at $\bar x\in \partial B'_1$: there exists
$c_0=c_0(n)>0$ such that
\begin{equation}\label{eq:asymp}
  \lim _{t\rightarrow 0+}\frac{\phi(\bar x +t\nu'_{\bar x})}{t^{s}}=c_0.
\end{equation}
Here $\nu'_{\bar x}$ is the in-plane outer unit normal of $A'_{1,2}$
at $\bar x$. Hence, we also have that
\begin{equation}
  \label{eq:asympt2}
  \frac{\phi(\bar x +t\nu'_{\bar x})}{t^{s}}> c_0/2\quad\text{for } 0<t<t_0 
\end{equation}
for sufficiently small $t_0$. Now let
$$ \psi(x):=c_{n,s}m_0 \phi\left(\frac{x-x_0}{\delta_0}\right),$$
and $D_{\delta_0, x_0}:=\{x: \frac{x-x_0}{\delta_0}\in D\}$. Note that
$D_{\delta_0, x_0}\subset B_1\setminus (\Omega_\epsilon\cap B'_1)$.
Applying the comparison principle in $D_{\delta_0, x_0}$ we have
\begin{equation}\label{eq:comparison}
  \psi(x)\leq u_\epsilon(x)-\epsilon, \quad x\in D_{\delta_0, x_0}.
\end{equation}
Choose now
$y_0\in \partial \Omega_\epsilon\cap \partial B_\delta(x_0)$ which
realizes the distance $\delta_0$. By \eqref{eq:comparison} and
recalling the explicit expression of $\psi$, we have
\begin{align}\label{eq:comparison2}
  \frac{c_{n,s}m_0\phi \left(\frac{y_0-x_0}{\delta_0}+\frac{t}{\delta_0}\nu'_{y_0}\right)}{t^s}\leq \frac{u_\epsilon(y_0+t\nu'_{y_0})-\epsilon}{t^s}.
\end{align}
We now want to use the estimate in Lemma~\ref{lem:harnack} to obtain
the bound on $m_0$. For that purpose, consider the following
rescalings at $y_0$
$$
{u}_{\epsilon,\epsilon^{1/s}}(x)=\frac{u_\epsilon(y_0+\epsilon^{1/s}x)}{\epsilon},
$$
which satisfy
$$
\div(|x_n|^{1-s}\nabla
{u}_{\epsilon,\epsilon^{1/s}})=2\beta({u}_{\epsilon,\epsilon^{1/s}})\H^{n-1}\LL\{x_n=0\}\quad\text{in
} B_{1/(2\epsilon^{1/s})}.
$$
Then, we can apply Lemma~\ref{lem:harnack} to conclude that
$|{u}_{\epsilon,\epsilon^{1/s}}|\leq C=C(n,s,A)$ in $B_{1/4}$. For the
function $u_\epsilon$ this translates into having the bound
$$
u_\epsilon(y_0+\epsilon^{1/s}x)\leq C\epsilon,\quad\text{for }|x|\leq
1/4.
$$
In particular, this gives that
$$
\frac{u_\epsilon(y_0+t\nu'_{y_0})-\epsilon}{t^s}\leq C,\quad\text{for
}t=\epsilon^{1/s}(\delta_0/4)
$$
Hence, from \eqref{eq:comparison2}, we obtain
$$
\frac{c_{n,s}m_0\phi
  \left(\frac{y_0-x_0}{\delta_0}+\tau\nu'_{y_0}\right)}{\tau^s}\leq
C\delta_0^s,\quad\text{for }\tau=\epsilon^{1/s}(1/4).
$$
Then, using \eqref{eq:asympt2}, we conclude that for small
$0<\epsilon<\epsilon_0$, necessarily
$$
m_0\leq C\delta_0^s
$$

\step{Step 2.}  We will show that for any
$y_0\in \Omega_\epsilon\cap B_{1/4}'$
\begin{equation}\label{eq:point_holder}
  |u_\epsilon(x)-u_\epsilon(y_0)|\leq C|x-y_0|^s
\end{equation}
for any $x\in B_1$, with a universal constant $C$.
\begin{enumerate}[(i)]
\item Suppose first $|x-y_0|\leq (1/8)\epsilon^{1/s}$.  For this case,
  recall that for the rescaling ${u}_{\epsilon,\epsilon^{1/s}}$ at
  $y_0$ defined in \emph{Step 1} above we have the estimate
  $\|{u}_{\epsilon,\epsilon^{1/s}}\|_{L^\infty(B_{1/4})}\leq C$. Then
  by Lemma~\ref{lem:frac-int-est-holder} we also have the estimate
  $\|{u}_{\epsilon,\epsilon^{1/s}}\|_{C^{0,s}(B_{1/8})}\leq C$, which
  then implies that
$$
|u_\epsilon(x)-u_\epsilon(y_0)|=\epsilon|{u}_{\epsilon,\epsilon^{1/s}}((x-y_0)/\epsilon^{1/s})-{u}_{\epsilon,\epsilon^{1/s}}(0)|\leq
C |x-y_0|^s
$$

\item Suppose now $x\in B_1'$ and $|x-y_0|\geq
  (1/8)\epsilon^{1/s}$. Then from \emph{Step 1} we have
  \begin{align*}
    |u_\epsilon(x)-u_\epsilon(y_0)|\leq 2\epsilon +C|x-y_0|^s\leq C|x-y_0|^s.
  \end{align*}
\end{enumerate}

Combining the estimates in (i)--(ii) above, we obtain that
\eqref{eq:point_holder} holds for any $x\in B_1'$. It remains to
establish \eqref{eq:point_holder} for $x\in B_1$ with $x_n\neq
0$. Note that it will be enough to show it for $x\in B_{1/2}^+$.

In order to do that, we first extend $u_\epsilon(\cdot,0)$ to all of
$\R^{n-1}$ by putting it equal to zero outside $B_1'$. Note that
estimate \eqref{eq:point_holder} will continue to hold now for all
$x\in\R^{n-1}$. Then, consider the convolution of the extended
$u_\epsilon(\cdot,0)$ with the Poisson kernel
$$P_{x_n}(x'):=C_{n,s}\frac{x_n^{2s}}{(|x'|^2+x_n^2)^{\frac{n-1+2s}{2}}}$$
for the operator $\L_s$. We then have
\begin{align*}
  &\left|(u_\epsilon(\cdot,0)\ast P_{x_n})(x') -u_\epsilon(y_0)\right|\\
  &\qquad=\left|\;\int_{\R^{n-1}}\left[u_\epsilon(x'-z',0)-u_\epsilon(y'_0,0)\right]P_{x_n}(z')
    dz'\right| \quad (\text{since $\textstyle
    \int _{\R^{n-1}}P_{x_n} =1$}, \ \forall x_n>0)\\
  &\qquad\leq C\int_{\R^{n-1}}|x'-y'_0-z'|^s P_{x_n}(z') dz' \quad (\text{by } \eqref{eq:point_holder})\\
  &\qquad\leq C \int_{\R^{n-1}}\left(|x'-y'_0|^s+|z'|^s\right)P_{x_n}(z')dz'\quad (\text{triangle inequality})\\
  &\qquad\leq C\left(|x'-y'_0|^s+|x_n|^s\right)\leq C|x-y_0|^s,
\end{align*}
where in the second last inequality we have used
$\int_{\R^{n-1}}|z'|^s P_{x_n}(z')dz'\leq C|x_n|^{s}$.  Next, the
difference
$$
v(x)=u_\epsilon(x)-\left(u_\epsilon(\cdot,0)\ast P_{x_n}\right)(x')
$$
satisfies
$$
\div(|x_n|^{1-2s}\nabla v)=0\quad\text{in }B_1^+,\quad
v=0\quad\text{on }B_1'.
$$
By making the odd reflection in $x_n$, we can make $v$ $\L_s$-harmonic
in $B_1$.  Hence, applying Lemma~\ref{lem:frac-int-est-holder}, we
will have
$$
\|v\|_{C^{0,s}(B_{1/2})}\leq C(n,s,L).
$$
Combining the estimates above, we then conclude that
\eqref{eq:point_holder} holds for all $x\in B_{1/2}$ and hence for all
$x\in B_1$.

\step{Step 3.} In this step, we complete the proof that
$u_\epsilon\in C^{0,s}(B_{1/8})$ uniformly in $\epsilon$. From
\emph{Step 2}, it is enough to show that for any
$x_1,x_2\in B_{1/8}\setminus \Omega_\epsilon$,
\begin{align*}
  |u_\epsilon(x_1)-u_\epsilon(x_2)|\leq C|x_1-x_2|^{s}.
\end{align*}
Let $d(x):=\dist(x,\Omega_\epsilon)$. Then consider the following two
subcases:
\begin{enumerate}[(a)]
\item Suppose that $|x_1-x_2|\leq
  \frac{1}{2}\max\{d(x_1),d(x_2)\}$.
  Without loss of generality, we assume that $d:=d(x_1)\geq
  d(x_2)$.
  Let also $y_1\in\Omega_\epsilon$ be such that $|x_1-y_1|=d$. Then,
  consider the rescaling of $u_\epsilon$ at $y_1$ by the factor of $d$
  \begin{align*} {u}_{\epsilon,d}(x):=\frac{u_\epsilon(y_1+d x)}{d^s}.
  \end{align*}
  From \emph{Step 2}, $0\leq {u}_{\epsilon,d}(x) \leq C|x|^s$ for
  $x\in B_1(\xi)$, $\xi:=(x_1-y_1)/d$.  Moreover, ${u}_{\epsilon,d}$
  satisfies the homogeneous equation
  $\div(|x_n|^{1-2s}\nabla {u}_{\epsilon,d})=0$ in $B_1(\xi)$.  By
  Lemma~\ref{lem:frac-int-est-holder},
  ${u}_{\epsilon,d}\in C^{0,s}(B_{3/4}(\xi))$. In particular, for
  $\eta:=\frac{x_2-y_1}{d}\in B_{3/4}(\xi)$ we have
  \begin{align*}
    |{u}_{\epsilon,d}(\eta)-{u}_{\epsilon,d}(\xi)|\leq C|\eta-\xi|^s.
  \end{align*}
  Rescaling back to $u_\epsilon$ we obtain
  \begin{align*}
    |u_\epsilon(x_1)-u_\epsilon(x_2)|\leq C|x_1-x_2|^{s}.
  \end{align*}
\item Suppose now $|x_1-x_2|>\frac{1}{2}\max\{d(x_1),d(x_2)\}$. In
  this case, by \emph{Step 2},
  \begin{align*}
    |u_\epsilon(x_1)-u_\epsilon(x_2)|&\leq C(d(x_1)^{s}+d(x_2)^{s})\leq C|x_1-x_2|^{s}.\qedhere
  \end{align*}
\end{enumerate}
\end{proof}

We conclude this section with the following remark that
$\{u_\epsilon\}$ are uniformly bounded also in
$W^{1,2}_{\loc}(B_1,|x_n|^{1-2s})$.
\begin{proposition}[Uniform $W^{1,2}$ bound]\label{prop:unif-W12}
  Let $u_\epsilon$ be a nonnegative solution of \eqref{eq:P_eps} with
  $\|u_\epsilon\|_{L^\infty(B_1)}\leq L$. Then
  $u_\epsilon\in W^{1,2}_{\loc}(B_1,|x_n|^{1-2s})$ for any
  $K\Subset B_1$ with
$$
\|u_\epsilon\|_{W^{1,2}(K,|x_n|^{1-2s})}\leq C(n,s,A,L,K)
$$
uniformly for all $\epsilon\in (0,1)$.
\end{proposition}
\begin{proof}
  Since $u_\epsilon$ is a nonnegative subsolution of $\L_s$, the proof
  follows from a standard energy inequality.
\end{proof}

\section{Passage to the limit as
  $\epsilon\rightarrow 0$}\label{s:limit}

\subsection{Compactness}

We start the section with the following local compactness lemma.
Recall that we always assume that the functions $u_\epsilon$ and $u$
are evenly extended in $x_n$-variable.

\begin{lemma}[Compactness and limit solutions]\label{lem:compactnesslemma}
  Let $u_\epsilon$ be a nonnegative solution to $(P_\epsilon)$. Then
  over a subsequence
  \begin{enumerate}[\upshape (i)]
  \item $\{u_\epsilon\}$ converges uniformly on compact subsets of
    $B_1$ to a function $u\in C^{0,s}_{\loc}(B_1)$.
  \item The limit function $u$ in (i) solves
    $\div(|x_n|^{1-2s} \nabla u)=0$ in $\{u>0\}$.
  \item
    $\beta_\epsilon(u_\epsilon)\stackrel{\ast}{\rightharpoonup} \mu$
    in the space of measures $\mathcal{M}(B_R')$ for any $0<R<1$.
  \item
    $|x_n|^{(1-2s)/2} \nabla u_\epsilon \rightarrow
    |x_n|^{(1-2s)/2}\nabla u$ strongly in $L^2_{\loc}(B_1)$.
  \item
    $\B_\epsilon (u_\epsilon)\stackrel{\ast}{\rightharpoonup} \chi$ in
    $L^\infty(B_1')$ for some $\chi\in L^\infty(B_1')$, where
    $\B_\epsilon$ are defined in \eqref{eq:Beps}.
  \end{enumerate}
  We call the function $u$ as above a \emph{limit solution} of
  \eqref{eq:freebdry}, and the pair $(u,\chi)$ a \emph{limit solution
    pair}.
\end{lemma}

\begin{proof}
  \begin{enumerate}[(i)]
  \item For any compact $K\Subset B_1$, we know by
    Theorem~\ref{thm:uniform_holder}, $u_\epsilon $ are uniformly
    bounded in $C^{0,s}(K)$. By Ascoli-Arzela's theorem up to a
    subsequence $\epsilon_j\rightarrow 0$, we obtain a function
    $u\in C^{0,s}(K)$ such that $u_{\epsilon_j}\rightarrow u$ in
    $C^{0,\alpha}(K)$ with $0<\alpha<s$. Since $u_\epsilon\geq 0$, it
    follows that $u\geq 0$.

  \item From (i) we know that $\{u>0\}$ is open. If $u(x_0)=c>0$, from
    the uniform convergence of $u_\epsilon$ we obtain a small
    neighborhood $U$ of $x_0$ such that
    $u_\epsilon (x)\geq c/2>\epsilon $ in $U$ for every
    $0<\epsilon\leq \epsilon_0 $ for some $\epsilon_0$ small. Then
    $u_\epsilon $ solves $\div(|x_n|^{1-2s} \nabla u_\epsilon)=0$ in
    $U$ for any $0<\epsilon\leq \epsilon_0$. The statement follows
    from the uniform convergence of $u_\epsilon$ to $u$.

  \item Since $u_\epsilon$ are uniformly bounded in
    $W^{1,2}_{\loc}(B_1,|x_n|^{1-2s})$, see
    Proposition~\ref{prop:unif-W12}, by plugging in a cut-off function
    into \eqref{eq:variational} it is straightforward to see that
    $\beta_\epsilon(u_\epsilon)$ are uniformly bounded in
    $L^1_{\loc}(B_1')$.

  \item For any $0<R<1$, plugging a test function
    $\phi = u_\epsilon \eta$ in \eqref{eq:variational}, where
    $\eta \in C^{\infty}_c(\R^n)$, $\eta\geq 0$ and $\eta =0$ outside
    $B_R$, we obtain that
    \begin{align}
      \int_{B_R}|x_n|^{1-2s}|\nabla u_\epsilon |^2\eta + |x_n|^{1-2s} u_\epsilon \nabla u_\epsilon \cdot\nabla \eta
      =-\int_{B'_R} 2\beta_\epsilon (u_\epsilon)u_\epsilon \eta \label{eq:strL2}
    \end{align}
    From (iii) and the fact that $\beta_\epsilon(u_\epsilon)$ is
    supported on the set $\{u_\epsilon\leq \epsilon\}$ we have
    \begin{equation}\label{eq:RHS}
      \text{RHS  of \eqref{eq:strL2}}\rightarrow 0\quad\text{as } \epsilon \rightarrow 0.
    \end{equation}
    From Proposition~\ref{prop:unif-W12}, we know that over a sequence
    $\epsilon=\epsilon_j\to 0$,
    $|x_n|^{(1-2s)/2}\nabla u_{\epsilon_j} $ converges to
    $|x_n|^{(1-2s)/2}\nabla u$ weakly in $L^2(B_R)$.  This together
    with \eqref{eq:strL2}, \eqref{eq:RHS} and the uniform convergence
    in (i) gives us
    \begin{equation}\label{eq:L2gradient1}
      \lim_{\epsilon_j \rightarrow 0}\int_{B_R} |x_n|^{1-2s}|\nabla u_{\epsilon_j} |^2\eta = -\int_{B_R}|x_n|^{1-2s} u \nabla u \cdot\nabla \eta. 
    \end{equation}
    On the other hand, for every $\delta >0$ consider the truncation
    $u^\delta=\max\{u-\delta, 0\}$. By (ii), $u^\delta$ solves the
    homogeneous equation in $\{u>\delta\}$. Taking the test function
    $\phi=u^\delta \eta$ in \eqref{eq:variational} for
    $\epsilon_j \in (0,\delta)$, where $\eta $ is the same as above,
    and letting $\epsilon_j\rightarrow 0$, we obtain
    \begin{equation}\label{eq:delta}
      0=\int_{B_R} |x_n|^{1-2s}|\nabla u^\delta |^2 \eta +\int_{B_R}|x_n|^{1-2s} u^\delta \nabla u^\delta \cdot\nabla \eta 
    \end{equation}
    Letting $\delta \rightarrow 0+$ in \eqref{eq:delta}, we obtain
    \begin{equation}\label{eq:L2gradient2}
      \int_{B_R} |x_n|^{1-2s} |\nabla u |^2 \eta =-\int_{B_R} |x_n|^{1-2s} u \nabla u \cdot\nabla \eta.
    \end{equation}
    Comparing \eqref{eq:L2gradient1} and \eqref{eq:L2gradient2}, we
    conclude
    \begin{equation}\label{eq:normconv}
      \lim_{\epsilon_j \rightarrow 0}\int_{B_R} |x_n|^{1-2s} |\nabla u_{\epsilon_j} |^2\eta=\int_{B_R} |x_n|^{1-2s} |\nabla u |^2 \eta. 
    \end{equation}
    This together with the weak $L^2$ convergence gives (iv).

  \item Since $0\leq \B_\epsilon(u_\epsilon) \leq M$, then there
    exists a subsequence $\B_{\epsilon_j}(u_{\epsilon _j})$ and $\chi$
    such that
    \[
    \B_{\epsilon_j}(u_{\epsilon _j}) \stackrel{\ast}{\rightharpoonup}
    \chi\quad \text{in } L^\infty(B_1').\qedhere
    \]
  \end{enumerate}
\end{proof}

\begin{lemma}\label{lem:L1_chi} Let $\chi$ be as in
  Lemma~\ref{lem:compactnesslemma}. Then
$$\chi \in \{0,M\}\quad\text{for a.e.\ }x\in B_1'.$$
\end{lemma}
\begin{proof}
  Given $0<\delta \ll M/4$ and $K\Subset B_1'$, let $\tilde{\delta}_1$
  and $\tilde{\delta}_2$ be the (unique) positive numbers satisfying
$$\int_0^{\tilde{\delta}_1}\beta(s)ds=\int_{1-\tilde{\delta}_2}^1 \beta(s)ds =\delta.$$
Then we have
\begin{equation}
  \label{eq:precom1}
  \begin{aligned}
    |K\cap \{\delta < \B_{\epsilon_j}(u_{\epsilon_j})<M-\delta\}|&=|K\cap \{\tilde{\delta}_1< \frac{u_{\epsilon_j}}{\epsilon_j}<1-\tilde{\delta}_2\}|\\
    & \leq |K\cap \{\beta_{\epsilon_j}(u_{\epsilon_j})\geq \frac{1}{\epsilon_j}\min_{[\tilde{\delta}_1,1-\tilde{\delta}_2]}\beta\}|\\
    &\leq\frac{\epsilon_j}{\min_{[\tilde{\delta}_1,1-\tilde{\delta}_2]}\beta}\int_{K}\beta_{\epsilon_j}(u_{\epsilon_j})\rightarrow
    0\quad \text {as } \epsilon_j\rightarrow 0,
  \end{aligned}
\end{equation}
where we have used Lemma~\ref{lem:compactnesslemma}(iii) and the
assumption that $\beta >0$ in $(0,1)$.  Hence if we let
$A_{\delta, K} := K\cap \{2\delta < \chi < M-2\delta\}$, then
\begin{align*}
  |A_{\delta,K}|&\leq |A_{\delta, K}\cap \{ \B_{\epsilon_j}(u_{\epsilon_j})\leq \delta \text{ or } \B_{\epsilon_j}(u_{\epsilon_j}) \geq M-\delta\}|\\
                &\qquad+|A_{\delta, K}\cap \{\delta < \B_{\epsilon_j}(u_{\epsilon_j})<M-\delta\}|\\
                &\leq |K\cap \{ |\B_{\epsilon_j}(u_{\epsilon_j})-\chi |\geq \delta \}|+|A_{\delta, K}\cap \{\delta < \B_{\epsilon_j}(u_{\epsilon_j})<M-\delta\}|.
\end{align*}    
By Lemma~\ref{lem:compactnesslemma}(v),
$\B_{\epsilon_j}(u_{\epsilon_j})\stackrel{\ast}{\rightharpoonup} \chi$
in $L^\infty(B_1')$, and moreover $0\leq \B_{\epsilon_j} \leq M$ for
all $j$, thus $\B_{\epsilon_j}(u_{\epsilon_j})\rightarrow \chi$ in
$L^1(B_1')$. This implies that
$|K\cap \{|\B_{\epsilon_j}(u_{\epsilon_j})-\chi |\geq \delta
\}|\rightarrow 0$
as $j\rightarrow \infty$. This combined with \eqref{eq:precom1} yields
that passing to the limit $j\rightarrow \infty$, $|A_{\delta, K}|=0$.
Because $\delta$ and $K$ are arbitrary, we have $\chi \in \{0,M\}$ for
a.e. $x\in B_1'$.
\end{proof}

The following lemma will play a crucial role in the paper. The proof
follows the lines of Lemma 3.2 in \cite{CLW} and is therefore omitted.

\begin{lemma}[Blowups at free boundary points]\label{lem:compactnesslemma2}
  Let $u_{\epsilon_j}\rightarrow u$ uniformly on compact subsets of
  $B_1$, and $\B_{\epsilon_j}\stackrel{\ast}{\rightharpoonup}\chi$ in
  $L^\infty(B_1')$, as in Lemma~\ref{lem:compactnesslemma}. For
  $x_0\in\F_u=\partial\{u(\cdot,0)>0\}\cap B_1'$ and $\lambda >0$,
  consider the following rescalings
  \begin{align*}
    u_{\lambda}^{x_0}(x)&:=\frac{1}{\lambda^{s}}u(x_0+\lambda x),\\
    u_{\epsilon,\lambda}^{x_0}(x)&:=\frac{1}{\lambda^{s}}u_{\epsilon}(x_0+\lambda x),\\
    \chi_{\lambda}^{x_0}(x')&:=\chi(x_0+\lambda x').
  \end{align*}
  Assume that there exists $\lambda_k\rightarrow 0$ such that
  $u_{\lambda_k}^{x_0}\rightarrow U$ as $k\rightarrow \infty$
  uniformly on compact subsets of $\R^n$ and
  $\chi_{\lambda_k}^{x_0}\stackrel{\ast}{\rightharpoonup} \chi_0$ in
  $L^\infty(\R^{n-1})$.  Then there exists $j(k)\rightarrow \infty$
  such that for every $j_k\geq j(k)$ we have that
  $(\epsilon_{j_k}/\lambda_k^{s})\rightarrow 0$ and
  \begin{enumerate}[\upshape (i)]
  \item $u_{\epsilon _{j_k},\lambda_k}^{x_0}\rightarrow U$ uniformly
    on compact subsets of $\R^{n}$
  \item
    $|x_n|^{(1-2s)/2}\nabla u_{\epsilon
      _{j_k},\lambda_k}^{x_0}\rightarrow |x_n|^{(1-2s)/2}\nabla U$
    in $L^2_{\loc}(\R^{n})$
  \item
    $\B_{\epsilon_{j_k}/\lambda_k^s}(u_{\epsilon_{j_k},\lambda_k}^{x_0})\stackrel{\ast}{\rightharpoonup}
    \chi_0$ in $L^\infty(\R^{n-1})$
  \item
    $|x_n|^{(1-2s)/2}\nabla u_{\lambda_k}^{x_0}\rightarrow
    |x_n|^{(1-2s)/2}\nabla U$ in $L^2_{\loc}(\R^{n})$.
  \end{enumerate}
\end{lemma}

We will call the function $U$ (or the pair $(U,\chi_0)$) a
\emph{blowup} of $u$ (or the pair $(u,\chi)$) at $x_0$. Note that the
above lemma says that $(U,\chi_0)$ is a limit solution pair on any
ball $B_R$, $R>0$.

\subsection{Solutions in the sense of domain variation}

We say that the function
$u_\epsilon\in W^{1,2}_{\loc}(B_1,|x_n|^{1-2s})$ is a
\emph{domain-variation solution} of $(P_\epsilon)$, if it satisfies
\begin{equation}\label{eq:domain-variation-eps}
  \int_{B_1}|x_n|^{1-2s}[-\nabla u_\epsilon \otimes \nabla u_\epsilon : \nabla \psi + \frac{1}{2}|\nabla u_\epsilon |^2 \div{\psi}] +\int_{B_1'} 2\B_\epsilon (u_\epsilon) \div{\psi} =0
\end{equation}
for every smooth vector field $\psi \in C^{\infty}_c(B_1; \R^n)$ with
$\psi(B_1')\subset\R^{n-1}$.  The name comes from the fact that the
equation \eqref{eq:domain-variation-eps} is equivalent to the
condition
$$\frac{d}{d\tau} J_\epsilon(u(x+\tau \psi(x)))\big|_{\tau=0}=0,$$
where
$$J_\epsilon(v)=\frac{1}{2}\int_{B_1} |x_n|^{1-2s}|\nabla v|^2 +
\int_{B_1'} 2\B_\epsilon (v)
$$
is the energy associated with \eqref{eq:P_eps}.  In particular, we see
that the weak solutions of \eqref{eq:variational} are also
domain-variation solutions.

Now, the advantage of the domain-variation solutions is as follows: if
$u_\epsilon$ is a weak solution of \eqref{eq:P_eps}, then by the
compactness Lemma~\ref{lem:compactnesslemma}, the limit solution pair
$(u,\chi)$ over any $\epsilon=\epsilon_j\to 0$ satisfies
\begin{equation}\label{eq:domain-variation}
  \int_{B_1}|x_n|^{1-2s}[-\nabla u \otimes \nabla u : \nabla \psi + \frac{1}{2}|\nabla u |^2 \div{\psi}] +\int_{B_1'} 2\chi  \div{\psi} =0,
\end{equation}
for every smooth vector field $\psi \in C^{\infty}_c(B_1; \R^n)$ with
$\psi(B_1')\subset\R^{n-1}$. While we could pass to the limit also in
the weak formulation \eqref{eq:variational}, the additional
information on $\chi$ that we have from Lemma~\ref{lem:L1_chi} will be
important in the sequel.

\subsection{Weiss-type monotonicity formula} In this section we prove
monotonicity formulas for the solution $u_\epsilon$ of
\eqref{eq:P_eps}, and the limit solution pair $(u,\chi)$ for
\eqref{eq:freebdry}.  This kind of formula has been first used by
Weiss \cite{Weiss} in the ``thick'' counterpart of our problem, as
well is the Alt-Caffarelli problem \cite{Weiss2}.

\begin{theorem}[Monotonicity formula for \eqref{eq:P_eps}]
  \label{thm:monotonicity}
  Let $x_0\in B_1'$ and $u_\epsilon$ be a solution to \eqref{eq:P_eps}
  with $\|u_\epsilon\|_{L^\infty(B_2)}\leq L$. For $0<r<1-|x_0|$, let
  \begin{align*}
    \Psi_\epsilon^{x_0}( u_\epsilon,r)&=\frac{1}{r^{n-1}}\int_{B_r(x_0)}
                                        |x_n|^{1-2s}|\nabla
                                        u_\epsilon|^2-\frac{s}{r^{n}}\int_{\partial
                                        B_r(x_0)}|x_n|^{1-2s}u_\epsilon^2\\
                                      &\qquad +\frac{1}{r^{n-1}}\int_{B'_r(x_0)}4\B_\epsilon (u_\epsilon).
  \end{align*}
  Then $r\mapsto \Psi_\epsilon^ {x_0}(u_\epsilon,r)$ is a
  nondecreasing function of $r$.
\end{theorem}
\begin{proof}
  For $0<r<1-|x_0|$ consider the rescalings
$$
u_{\epsilon,r}(x):=\frac{u_\epsilon (x_0+rx)}{r^{s}}.
$$
Then
\begin{equation*}
  \Psi^{x_0}_\epsilon (u_\epsilon,r)=\Psi^{0}_\epsilon(u_{\epsilon,r},1)=\int_{B_1} |x_n|^{1-2s} |\nabla u_{\epsilon,r}|^2-s\int_{\partial B_1} |x_n|^{1-2s}u_{\epsilon,r}^2 + \int_{B'_1}4\B_\epsilon (r^{s}u_{\epsilon,r}).
\end{equation*}
Thus
\begin{align*}
  \frac{d}{dr}\Psi^{x_0}_\epsilon (u_\epsilon,r)& = \int_{B_1}2|x_n|^{1-2s}\nabla u_{\epsilon,r}\cdot \nabla \frac{d}{dr}u_{\epsilon,r}-2s\int_{\partial B_1} |x_n|^{1-2s}u_{\epsilon,r}\frac{d}{dr} u_{\epsilon,r}\\
                                                &\qquad +\int_{B'_1}4\beta_\epsilon(r^{s}u_{\epsilon,r})\left(sr^{s-1}u_{\epsilon,r}+r^{s}\frac{d}{dr}u_{\epsilon,r}\right).
\end{align*}
Now, noting that $u_{\epsilon,r}$ solves
\begin{align*}
  \div(|x_n|^{1-2s}\nabla u_{\epsilon,r})
  &=2r^{s}\beta_\epsilon(r^{s}u_{\epsilon,r})\H^{n-1}\LL{\{x_n=0\}}\quad\text{in
    } B_{1/r}
\end{align*}
and integrating by parts, using also the nonnegativity of
$\beta_\epsilon(u_\epsilon)u_\epsilon$, we have
\begin{align*}
  \frac{d}{dr}\Psi^{x_0}_\epsilon &=2\int_{\partial B_1}|x_n|^{1-2s}\left(\partial_\nu u_{\epsilon,r}-su_{\epsilon,r}\right)\frac{d}{dr}u_{\epsilon,r}+\int_{B'_1}4\beta_\epsilon(r^{s}u_{\epsilon,r})sr^{s-1}u_{\epsilon,r}\\
                                  &\geq  2\int_{\partial B_1}|x_n|^{1-2s}\left(\partial_\nu u_{\epsilon,r}-su_{\epsilon,r}\right)\frac{d}{dr}u_{\epsilon,r}.
\end{align*}
Observing that for $x\in \partial B_1$,
\begin{align*}
  \frac{d}{dr}u_{\epsilon,r}(x)&=r^{-(1+s)}\left((rx)\cdot\nabla u_\epsilon (x_0+rx)-su_\epsilon(x_0+rx)\right),\\
  \partial_\nu u_{\epsilon,r}(x)&=x\cdot \nabla u_{\epsilon,r}(x)=r^{-s}(rx)\cdot \nabla u_\epsilon(x_0+rx),
\end{align*}
we then obtain
\begin{align*}
  \frac{d}{dr}\Psi^{x_0}_\epsilon(u_\epsilon,r)\geq \frac{2}{r^{n+1}}\int_{\partial B_r(x_0)}|x_n|^{1-2s}\left((x-x_0)\cdot \nabla u_{\epsilon}-su_\epsilon\right)^2\geq 0.
\end{align*}
This implies that $r\mapsto \Psi^{x_0}_\epsilon(u_\epsilon,r)$ is
monotonically nondecreasing.
\end{proof}

By the compactness Lemma~\ref{lem:compactnesslemma}, passing to the
limit in a subsequence $\epsilon _j$ we get the following monotonicity
formula for the limit pair $(u, \chi)$. Similar monotonicity formula
was used in the thin and fractional Alt-Caffarelli problems in
\cite{AP} and \cite{Allen}, respectively.

\begin{theorem}[Monotonicity formula for \eqref{eq:freebdry}]
  \label{thm:monotonicity2}
  Let $(u,\chi)$ be a limit solution pair and $x_0\in B_1'$.  For
  $0<r<1-|x_0|$, let
  \begin{align*}
    \Psi^{x_0}(u,r)&=\frac{1}{r^{n-1}}\int_{B_r(x_0)} |x_n|^{1-2s}|\nabla
                     u|^2-\frac{s}{r^{n}}\int_{\partial
                     B_r(x_0)}|x_n|^{1-2s}u^2\\
                   &\qquad +\frac{1}{r^{n-1}}\int_{B'_r(x_0)}4\chi.
  \end{align*}
  Then $r\mapsto \Psi^{x_0}(u,r)$ is monotonically nondecreasing. More
  precisely, for $0<\rho<\sigma <1-|x_0|$,
  \begin{align*}
    \Psi^{x_0}(u,\sigma)-\Psi^{x_0}(u,\rho)\geq \int_{\rho}^{\sigma}\frac{2}{r^{n+1}}\int_{\partial B_r(x_0)}|x_n|^{1-2s}\left((x-x_0)\cdot \nabla u-su\right)^2 d\sigma_r dr\geq 0.
  \end{align*}
  In particular, the limit
  $\Psi^{x_0}(u,0+)=\lim_{r\to 0+}\Psi^{x_0}(u,r)$ exists.
\end{theorem}

We will call the quantity $\Psi^{x_0}(u,0+)$ the \emph{Weiss energy of
  $u$ at $x_0$}.

\medskip Next we prove a corollary of the monotonicity formula
above. For notation convenience, sometimes we write the dependence of
$\Psi$ on $(u,\chi)$ explicitly, i.e.\ we write $\Psi(u,\chi, r)$
instead of $\Psi(u,r)$.

\begin{corollary}\label{cor:homoblowup}
  Let $(u,\chi)$ be a limit solution pair and
  $x_0\in \F_u=\partial\{u(\cdot,0)>0\}\cap B_1'$. Then
  \begin{enumerate}[\upshape (i)]
  \item $\Psi^{x_0}(u, r)\geq -C$ for any $r>0$.
  \item Suppose for $\lambda_j\rightarrow 0$ a blowup sequence
    $(u_{\lambda_j}, \chi_{\lambda_j})$ (as in
    Lemma~\ref{lem:compactnesslemma2}) satisfies
    \begin{align*}
      u_{\lambda_j}\rightarrow u_0&\quad\text{uniformly on compact subsets of } \R^{n},\\
      \chi_{\lambda_j}\stackrel{\ast}{\rightharpoonup} \chi_0&\quad\text{in } L^\infty(\R^{n-1}).
    \end{align*}
    Then $\Psi^0(u_0,\chi_0,r)$ is constant in $r$. Moreover,
$$
\Psi^{0}(u_0,\chi_0, r) = \Psi^{x_0}(u,\chi,0+)= \int_{B'_1}4\chi_0.
$$
In particular, $\Psi^{x_0}(u,0+)\geq 0$.

\item $u_0$ is a homogeneous function of degree $s$, i.e.
 $$u_0(\lambda x)=\lambda ^{s}u_0(x), \quad \text{for a.e. } x\in \R^n, \lambda \geq 0.$$
\end{enumerate}
\end{corollary}

\begin{proof}
  \begin{enumerate}[(i)]
  \item Since $u\in C^{0,s}_{\loc}(\R^n)$ and $u(x_0)=0$, then
$$
\frac{s}{r^{n}}\int_{\partial B_r(x_0)}|x_n|^{1-2s}u^2 \leq C,
$$
uniformly in $r$. Thus $\Psi^{x_0}(u,r)\geq-C$ for any $r>0$.

\item We use the following scaling property of $\Psi$:
  \begin{equation}\label{eq:rescalingblowup}
    \Psi^{0}(u_{\lambda_j},\chi_{\lambda_j},r)=\Psi^{x_0}(u,\chi, \lambda_jr).
  \end{equation}
  From Theorem~\ref{thm:monotonicity2}, the limit
  $\Psi^{x_0}(u,\chi,\lambda_jr)$ as $j\rightarrow \infty$ exists, and
  is equal to $\Psi^{x_0}(u,\chi,0+)$. Thus, passing to
  $j\rightarrow \infty$ in \eqref{eq:rescalingblowup} we have
  \begin{align*}
    \Psi^0(u_0,\chi_0,r)=\Psi^{x_0}(u,\chi,0+),\quad \text{for all }0<r<1.
  \end{align*}

\item By (ii) and Theorem~\ref{thm:monotonicity2}, for all
  $0<s_1<s_2<1$,
  \begin{align*}
    \Psi^{x_0}(u,\chi, \lambda_j s_1)-\Psi^{x_0}(u,\chi, \lambda_j s_2)
    & = \Psi^{0}(u_{\lambda_j},\chi_{\lambda_j},s_1)-\Psi^{0}(u_{\lambda_j},\chi_{\lambda_j},s_2)\\
    & \geq \int_{s_1}^{s_2}\frac{2}{r^{n+1}}\int_{\partial B_r}|x_n|^{1-2s}\left(x\cdot \nabla u_{\lambda_j}-su_{\lambda_j}\right)^2 d\sigma_r dr
  \end{align*}
  By Theorem~\ref{thm:monotonicity2}, the left hand side goes to zero
  as $j\rightarrow \infty$. Thus passing to the limit we have
  \begin{align*}
    &\lim_{j\rightarrow\infty} \int_{s_1}^{s_2}\frac{2}{r^{n+1}}\int_{\partial B_r}|x_n|^{1-2s}\left(x\cdot \nabla u_{\lambda_j}-su_{\lambda_j}\right)^2 d\sigma_r dr\\
    &\qquad =\int_{s_1}^{s_2}\frac{2}{r^{n+1}}\int_{\partial B_r}|x_n|^{1-2s}\left(x\cdot \nabla u_{0}-su_{0}\right)^2 d\sigma_r dr=0, \quad \text{for all }0<s_1<s_2<1.
  \end{align*}
  This yields the desired homogeneity of $u_0$.

  Finally, we show that
  $\Psi^{0}(u_0,\chi_0,r)=\int_{B'_1}4\chi_0$. In fact,
  \begin{align*}
    \Psi^{0}(u_0,\chi_0,r)= \int_{B_1}|x_n|^{1-2s}|\nabla u_0|^2-s\int_{\partial B_1}u_0^2 +\int_{B'_1}4\chi_0
  \end{align*} 
  By (ii), $u_0$ is homogeneous of degree $s$ and satisfies
  $u_0 \div(|x_n|^{1-2s}\nabla u_0)=0$, since
  $\beta_\epsilon(u_\epsilon)\in L^1_{\loc}(\R^{n-1}\times \{0\})$ and
  $u_\epsilon \beta_\epsilon(u_\epsilon)\rightarrow 0$ in
  $L^1_{\loc}(\R^{n-1}\times \{0\})$ as $\epsilon\rightarrow 0$. Then
$$\int_{B_1}|x_n|^{1-2s} |\nabla u_0|^2-s\int_{\partial B_1}u_0^2 =0$$
and consequently
\[
\Psi^{0}(u_0,\chi_0,r)=\int_{B'_1}4\chi_0\geq 0.\qedhere
\]
\end{enumerate}
\end{proof}

\section{Asymptotic behavior of limit solutions}\label{s:asymptotic}

Assume $(u,\chi)$ is a limit solution pair in the sense of
Lemma~\ref{lem:compactnesslemma}.  In this section we will show the
asymptotic behavior of $u$ around the regular free boundary points of
$\mathcal{F}_u=\partial\{u(\cdot,0)>0\}\cap B_1'$. By \emph{regular
  free boundary point} we mean a point $x_0\in \mathcal{F}_u$, where
$\mathcal{F}_u$ has an \emph{inward unit normal $\nu$ in the
  measure-theoretic sense}. More precisely, by this we understand
$\nu\in\R^{n-1}$, $|\nu|=1$, such that
\begin{equation}\label{eq:normal}
  \lim_{r\rightarrow 0} \frac{1}{r^{n-1}}\int_{B'_r(x_0)}\left|\chi_{\{u>0\}}-\chi_{\{(x',0):\left\langle x'-x'_0,\nu\right\rangle>0\}}\right|dx' =0.
\end{equation}

The main result of this section is as follows.

\begin{theorem}[Limit solutions at regular points]\label{thm:maintheorem}
  Let $(u,\chi)$ be a limit solution pair in the sense of
  Lemma~\ref{lem:compactnesslemma}. Let $x_0\in \mathcal{F}_u$ be such
  that
  \begin{enumerate}[\upshape (i)]
  \item $\mathcal{F}_u$ has at $x_0$ an inward unit normal $\nu$ in
    the measure theoretic sense,
  \item $u$ is nondegenerate at $x_0$ in the sense that there exist
    $c, r_0>0$ such that
    \begin{equation}\label{eq:nondegenerate}
      \frac{1}{r^{n-1}}\int_{B'_r(x_0)} u dx' \geq c\,r^{s}, \quad \text{for any } 0<r<r_0.
    \end{equation}
  \end{enumerate}
  Then we have the following asymptotic development
  \begin{equation*}
    u(x)=\sqrt{\frac{2M}{c_0(s)}}2^{-s}\left( \sqrt{\left\langle x'-x'_0, \nu \right\rangle ^2 +x_n^2}+\left\langle x'-x'_0, \nu \right\rangle\right)^{s} + o (|x-x_0|^{s}),
  \end{equation*}
  with
  $$c_0(s)=s^22^{-1-2s}\frac{\sqrt{\pi}(7+4s(s-2))\Gamma(1-s)}{\Gamma(\frac{7}{2}-s)}.
$$
Moreover, we have $\Psi^{x_0}(u,\chi, 0+)=2M|B'_1|$.
\end{theorem}
\begin{remark}
  In the case $s=\frac{1}{2}$, $c_0(\frac{1}{2})=\frac{\pi}{8}$.
\end{remark}

We start by identifying the limit solutions $u$ of the form
$$
u(x)=\frac{\alpha}{2^s} \left(\sqrt{x_1^2+x_n^2}+x_1\right)^s
$$
for some $\alpha >0$.  The free boundary in this case is
$\F_u=\{x_1=0,x_n=0\}$. This is the first step of understanding the
asymptotic development of limit solutions around the `regular' free
boundary points.

\begin{proposition}\label{prop:basicex}
  Let $(u,\chi)$ be a limit solution pair as in
  Lemma~\ref{lem:compactnesslemma} such that
$$
u=\frac{\alpha}{2^s}\left(\sqrt{x_1^2+x_n^2}+x_1\right)^{s} \text{ for
  some } \alpha > 0.$$ Then
$$
\chi=M\chi_{\{x_1>0\}}\quad\H^{n-1}\text{-a.e.\ in }\R^{n-1}
$$
and the constant $\alpha $ is given by
$$
\alpha = \sqrt{\frac{2M}{c_0(s)}},\quad\text{with $c_0(s)$ as in
  Theorem~\ref{thm:maintheorem}}.
$$
\end{proposition}

\begin{proof}
  In the proof below, we identify $\R^{n-2}$ with
  $\{0\}\times\R^{n-2}$, and denote $x''=(x_2,\ldots, x_{n-1})$,
  $B_r''=B_r'\cap\{x_1=0\}$.

  \step{Step 1.} We will show that for any
  $\phi\in C_c^\infty(B_1';\R^{n-1})$,
  $\phi=(\phi_1,\ldots, \phi_{n-1})$,
  \begin{align}\label{eq:alpha01}
    c_0(s)\alpha^2\int_{B_1''}\phi_1(0,x'')dx''+\int_{B_1'}2\chi \div\phi=0.
  \end{align}
  If we take $\psi \in C_c^\infty(B_1; \R^{n})$, with
  $\psi = (\phi,0)$ on $B_1'$ as a test function in
  \eqref{eq:domain-variation}, then
$$
\int_{B_1} \left[\frac{1}{2}|\nabla u |^2 \div{(|x_n|^{1-2s}\psi)} -
  (\nabla u )^t D \psi \nabla u |x_n|^{1-2s}\right] +\int_{B_1'} 2\chi
\div{\phi}=0.
$$
Since $|x_n|^{\frac{1-2s}{2}}|\nabla u|\in L^2_{\loc}(\R^n)$ by
Lemma~\ref{lem:compactnesslemma} (iv) and $|\psi_n(x)|\leq C|x_n|$,
then for any small $\delta>0$,
\begin{multline}\label{eq:variation2}
  \int_{\sqrt{x_1^2+x_n^2}\geq \delta}\left[\frac{1}{2}|\nabla u |^2
    \div{(|x_n|^{1-2s}\psi)} - (\nabla u )^t D \psi \nabla u
    |x_n|^{1-2s}\right]\\+\int_{B_1'} 2\chi \div{\phi}=O(\delta).
\end{multline}
We next rewrite the first integrand in the LHS of
\eqref{eq:variation2} as follows.
\begin{multline}\label{eq:div1}
  \frac{1}{2}|\nabla u|^2 \div {(|x_n|^{1-2s}\psi)}-(\nabla u)^tD\psi
  \nabla u|x_n|^{1-2s}\\
  =\frac{1}{2}\div{(|\nabla u|^2 \psi |x_n|^{1-2s})}-\frac{1}{2}\nabla
  (|\nabla u|^2)\cdot \psi
  |x_n|^{1-2s}\\
  -(\nabla u)^tD\psi \nabla u|x_n|^{1-2s}.
\end{multline}
To estimate the last two terms above, we observe that
\begin{align*}
  (\nabla u\cdot \psi)\div(|x_n|^{1-2s}\nabla u)=0\quad\text{in }
  \R^n\setminus \left\{\sqrt{x_1^2+x_n^2}<\delta\right\},
\end{align*}
since $\div{(|x_n|^{1-2s} \nabla u)}=0$ in
$\R^n\setminus \{x_1\leq 0,x_n=0\}$ and $\psi\cdot e_n=0$,
$\nabla' u =0$ on $\{x_1<0,x_n=0\}$. Thus, in
$\{\sqrt{x_1^2+x_n^2}\geq \delta\}$,
\begin{multline}\label{eq:div2}
  \frac{1}{2}\nabla (|\nabla u|^2)\cdot \psi |x_n|^{1-2s}+(\nabla u)^tD\psi \nabla u|x_n|^{1-2s}\\
  =\div{((\nabla u\cdot \psi)\nabla u|x_n|^{1-2s}}).
\end{multline}
Combining \eqref{eq:div1} and \eqref{eq:div2}, we rewrite
\eqref{eq:variation2} as
\begin{multline*}
  \int_{\sqrt{x_1^2+x_n^2}\geq \delta}\frac{1}{2}\div{(|\nabla u|^2
    \psi |x_n|^{1-2s})}-\div{((\nabla u\cdot \psi)\nabla
    u|x_n|^{1-2s}})\\+\int_{B_1'} 2\chi \div{\phi}=O(\delta).
\end{multline*}
An integration by parts gives
\begin{multline}
  \int_{\sqrt{x_1^2+x_n^2}=\delta}\frac{1}{2}|\nabla u|^2 (\psi\cdot
  \nu) |x_n|^{1-2s}-\int_{\sqrt{x_1^2+x_n^2}=\delta}(\nabla u\cdot
  \psi)(\nabla u \cdot \nu)|x_n|^{1-2s}\\+\int_{B_1'}2\chi \div
  \phi=O(\delta ),
\end{multline}
where $\nu$ is the unit outer normal to
$\{\sqrt{x_1^2+x_n^2}=\delta\}$. Next we estimate the first two
integrals in the LHS of the above equality. Since $u$ is homogeneous
of degree $s$ (and depends only on $x_1$ and $x_n$),
$\nabla u\cdot \nu=su$. Thus,
\begin{align*}
  \int_{\sqrt{x_1^2+x_n^2}=\delta}(\nabla u\cdot \psi)(\nabla u \cdot \nu)|x_n|^{1-2s}=\int_{\sqrt{x_1^2+x_n^2}=\delta}(\nabla u\cdot \psi)(su)|x_n|^{1-2s}=O(\delta),
\end{align*}
where we have used the growth estimate of $u$ around the zero set as
well as $|\nabla u||x_n|^{\frac{1-2s}{2}}\in L^2_{\loc}(\R^{n})$.  To
estimate the first integral we use the polar coordinates in
$(x_1,x_n)$-plane: $x_1=r\cos(\theta)$, $x_n=r\sin(\theta)$. Then on
$\{\sqrt{x_1^2+x_n^2}=\delta\}$,
\begin{align*}
  &\frac{1}{2}|\nabla u|^2(\psi\cdot \nu)|x_n|^{1-2s}\\
  &\qquad =\frac{s^2\alpha^2}{2\delta}(\cos(\theta/2))^{-1+2s}(|\sin(\theta)|)^{1-2s}\left[\cos(\theta)(\psi\cdot e_1)+\sin(\theta)(\psi\cdot e_n)\right]
\end{align*}
Since $\psi\cdot e_n=0$ on $B_1'$,
\begin{align*}
  &\lim_{\delta\rightarrow 0}\int_{\sqrt{x_1^2+x_n^2}=\delta}\frac{1}{2}|\nabla u|^2(\psi\cdot \nu)|x_n|^{1-2s} dx=\lim_{\delta\rightarrow
    0}\int_{B_1''}\frac{s^2\alpha^2}{2\delta}\cdot 2\delta \times\\
  &\qquad\times \int_0^{\pi} (\cos(\theta/2))^{-1+2s}(\sin(\theta))^{1-2s}(\cos(\theta))^2\psi_1(\delta\cos(\theta),x'',\delta\sin(\theta))\ d\theta dx''\\
  &\qquad=c_0(s)\alpha^2\int_{B_1''}\psi_1(0,x'',0)dx'',
\end{align*}
where
\begin{align*}
  c_0(s)&=s^2\int_0^\pi (\cos(\theta/2))^{-1+2s}(\sin(\theta))^{1-2s}(\cos(\theta))^2 d\theta \\
        &=s^22^{-1-2s}\frac{\sqrt{\pi}(7+4s(s-2))\Gamma(1-s)}{\Gamma(\frac{7}{2}-s)}.
\end{align*}
Combining the above estimates and letting $\delta\rightarrow 0$, we
have for all $\phi\in C^\infty_c(B_1';\R^{n-1})$
\begin{align*}
  c_0(s)\alpha^2\int_{B_1''}\psi_1(0,x'',0)dx''+\int_{B_1'}2\chi \div\phi=0
\end{align*}
Recalling $\psi_1=\phi_1$ on $B_1'$, we have proved
\eqref{eq:alpha01}.

\step{Step 2.} We now show that $\alpha=\sqrt{\frac{2M}{c_0(s)}}$ and
$\chi=M\chi_{\{x_1>0\}}$ $\H^{n-1}$-a.e.\ in $\R^{n-1}$.
\begin{enumerate}[(a)]

\item $\chi\equiv M$ in $\{x_1>0\}$.

  In fact, if $y\in \{x_1>0\}$, then $u(y',0)= \alpha
  (y_1)^s>0$.
  Hence by the uniform H\"older convergence of $u_{\epsilon_j}$ to
  $u$, we have $u_{\epsilon _j}\geq \epsilon _j$ in a neighborhood of
  $(y',0)$ for any $j\geq j_0$ for some $j_0=j_0(\alpha, y_1)$ large
  enough. Thus for $j>j_0$ and $(x', 0)$ in the neighborhood of
  $(y',0)$ we have
$$ \B_{\epsilon _j}(u_{\epsilon _j})(x')=\int_0^{u_{\epsilon_j(x')/\epsilon _j}}\beta (s)ds =M.$$
Letting $j\rightarrow \infty$ and since $y'$ is arbitrary, we get
$\chi\equiv M$ in $\{x_1>0\}$.

\item $\chi\equiv 0$ in $\{x_1<0\}$.

  In fact, we can take any $\phi$ in \eqref{eq:alpha01} such that
  $\supp{\phi_1}\subset \R^{n-1}\cap\{x_1<0\}$. Then the LHS of
  \eqref{eq:alpha01} will vanish. This implies that $\chi=const$ in
  $\{x_1<0\}$. By Lemma~\ref{lem:L1_chi}, $\chi\equiv M$ or
  $\chi\equiv 0$ in $\{x_1<0\}$. If $\chi\equiv M$ in $\{x_1<0\}$,
  then $\chi\equiv M$ in $\R^{n-1}$ by (a). Thus from
  \eqref{eq:alpha01} $\int_{\R^{n-2}}\phi_1(0,x'')dx''=0$ for any
  compactly supported vector field $\phi$, which is a contradiction.
\end{enumerate}

This then implies the claim that $\chi=M\chi_{\{x_1>0\}}$. Next,
applying an integration by parts to the RHS of \eqref{eq:alpha01} we
have
\begin{equation}\label{eq:deltaestimate3}
  c_0(s)\alpha^2\int_{B_1''}\phi_1(0,x'') dx''
  = 2M\int_{B_1''}\phi_1(0,x'') dx''.
\end{equation} 
This implies
\[
\alpha=\sqrt{\frac{2M}{c_0(s)}}.\qedhere
\]

\end{proof}

We are now ready to prove Theorem~\ref{thm:maintheorem}.

\begin{proof}[Proof of Theorem~\ref{thm:maintheorem}]
  Without loss of generality we assume $x_0=0$ and $\nu = e_1$. We
  also extend $u$ by the even reflection with respect to $x_n$.

  Consider the rescalings $u_\lambda$ and $\chi_\lambda$ at the
  origin. By Theorem~\ref{thm:uniform_holder} given $\rho>0$,
  $u_\lambda $ is uniformly bounded in
  $ C^{0,s}(\overline{B_{\rho/\lambda}}) $. Therefore, there exists a
  sequence $\lambda_j\rightarrow 0$, $u_0\in C^{0,s}(\R^n) $ and
  $\chi_0 \in L^\infty (\R^n)$ such that
  $u_{\lambda_j}\rightarrow u_0 $ uniformly on compact subsets of
  $\R^n$ and $\chi_{\lambda_j} \stackrel{\ast}{\rightharpoonup}\chi_0$
  in $L^\infty (\R^{n-1})$.

  Now, rescaling \eqref{eq:normal}, we see that for every $R>0$,
  \begin{equation*}
    |\{u_\lambda (\cdot,0)>0\}\cap \{x_1<0\}\cap B'_R|\rightarrow 0\quad\text{as } \lambda\rightarrow 0,
  \end{equation*}
  and we deduce that $u_0=0$ $\H^{n-1}$-a.e.\ in $\{x_1<0,
  x_n=0\}$.
  By continuity, $u_0$ vanishes on all of $\{x_1\leq 0, x_n=0\}$.
  Besides, we readily have that $u_0$ satisfies
  $\div(|x_n|^{1-2s}\nabla u_0)=0$ in $\{u_0>0\}$. Thus, we can apply
  Corollary \ref{cor:asymptotic2} in Appendix to obtain an asymptotic
  development
  \begin{equation}\label{eq:asymp2}
    u_0(x)=\alpha 2^{-s}\left((x_1^2+x_n^2)^{1/2}+x_1\right)^{s} +o(|x|^{s})  
  \end{equation}
  with $\alpha\geq 0$. Next, note that by rescaling the nondegeneracy
  condition \eqref{eq:nondegenerate} and passing to the limit, we have
  \begin{align*}
    \frac{1}{r^{n-1}}\int_{B'_r} u_0 dx' \geq c\,r^{s} , \quad \text{for any }r>0,
  \end{align*}
  which implies that $\alpha>0$.  On the other hand, by
  Corollary~\ref{cor:homoblowup}(iii), $u_0$ is homogeneous of degree
  $s$ and hence
  \begin{equation*}
    u_0(x)=\alpha 2^{-s}\left((x_1^2+x_n^2)^{1/2}+x_1\right)^{s}.
  \end{equation*}
  Now by Lemma~\ref{lem:compactnesslemma2}, $(u_0,\chi_0)$ is a limit
  solution pair and we can apply Proposition~\ref{prop:basicex} to
  conclude that $\alpha=\sqrt{\frac{2M}{c_0(s)}}$, which is a constant
  independent of the sequence $\lambda_j$. This proof the asymptotic
  development for $u$.

  Finally, by Proposition~\ref{prop:basicex}, we also have that
  $\chi_0= M\chi_{\{x_1>0\}}$ and hence by
  Corollary~\ref{cor:homoblowup}(ii),
  \[
  \Psi^0(u,\chi, 0+)=\int_{B'_1}4\chi_0=2M|B'_1|.\qedhere
  \]
\end{proof}

We conclude the paper with two propositions regarding some additional
properties of limit solutions and the value of
$\Psi^{x_0}(u,\chi,0+)$, which may be useful in the further treatment
of the free boundary. At the end we also give an alternative proof of
Theorem~\ref{thm:maintheorem}, relying on these results, rather than
on asymptotic developments in Appendix.

The first proposition says that the nondegeneracy condition
\eqref{eq:nondegenerate} at $x_0\in \F_u$ implies also a nondegeneracy
for $\{u(\cdot,0)=0\}$. The proof uses the dimension reduction
argument (see e.g.\ \cite{Weiss}).

\begin{proposition}\label{prop:classify}
  Let $(u,\chi)$ be a limit solution pair, and $x_0\in \mathcal{F}_u$
  such that $u$ is nondegenerate at $x_0$ in the sense of
  \eqref{eq:nondegenerate}. Then
  $0\leq \Psi^{x_0}(u,\chi,0+)< 4M|B'_1|$. In particular,
  $|B_r'(x_0)\cap\{u(\cdot,0)=0\}|>0$ for any $r>0$.
\end{proposition}
\begin{proof} Without loss of generality we assume that $x_0=0$. By
  Corollary~\ref{cor:homoblowup}(ii) and the inequality
  $0\leq \chi_0\leq M$, to prove the first part of the proposition, we
  essentially have to exclude the possibility that
  $\Psi^{0}(u,\chi,0+)=4M|B'_1|$, which is equivalent to having
  $\chi_0=M$ a.e.\ on $\R^{n-1}$ (recall that $(u_0,\chi_0)$ is a
  blow-up limit at $0$ along a subsequence
  $(u_{\lambda_j}, \chi_{\lambda_j})$).  We want to show that this
  implies that $u_0$ depends only on one variable $x_n$.

  To this end, take $\hat x_0\in \mathcal{F}_{u_0}\subset \R^{n-1}$
  such that $\hat x_0\neq 0$. Note that such point exists, otherwise
  $u_0>0$ in $\R^n\setminus \{0\}$ ($u_0\equiv 0$ on $\R^{n-1}$ is
  excluded by the nondegeneracy assumption), which would imply that
  $\div{(|x_n|^{1-2s}\nabla u_0)}=0$ in $\R^n\setminus \{0\}$.  Since
  $u_0$ is locally bounded, by the removability of point
  singularities\footnote{\emph{Proof.} Suppose $|u_0|\leq L$ and let
    $w$ be such that $\L_sw=0$ in $B_1$ and $w=u_0$ on $\partial
    B_1$.
    By comparing the difference $w-u_0$ with
    $\frac{2L}{\Phi_s(\delta)}\Phi_s(x)$ in $B_1\setminus B_{\delta}$,
    where $\Phi_s(x)=C_s|x|^{-(n-1-2s)}$ is the fundamental solution
    of $\L_s$, and letting $\delta\to 0+$, we conclude that
    $w=u_0$.\qed} we would have that
  $\div{(|x_n|^{1-2s}\nabla u_0)}=0$ in all of $\R^n$. Then, by the
  Harnack inequality \cite{FKS} (which implies Liouville theorem) we
  would have that $u_0\equiv 0$, which is a contradiction. Since $u_0$
  is homogeneous, then $\lambda \hat x_0\in \mathcal{F}_{u_0}$ for
  each $\lambda>0$. Moreover, due to the equality $\chi_0=M$ a.e.\ in
  $\R^{n-1}$ we will have
  $\Psi^{\lambda \hat x_0}(u_0,\chi_0, 0+)=4M|B'_1|$ for each
  $\lambda>0$. Applying the monotonicity formula
  (Theorem~\ref{thm:monotonicity2}) to $u_0$ at $\hat x_0$ we get
  \begin{equation}\label{eq:dimension}
    \begin{split}
      &\int_0^R \frac{2}{r^{n+1}}\int_{\partial B_r(\hat x_0)}|x_n|^{1-2s}\left( (x-\hat x_0)\cdot \nabla u_0-s u_0\right)^2 d\sigma_r dr \\
      &\qquad\leq \Psi^{\hat x_0}(u_0, \chi_0, R)-\Psi^{\hat x_0}(u_0, \chi_0, 0+)\\
      &\qquad= \Psi^{\hat x_0}(u_0, \chi_0, R)-\Psi^{0}(u_0, \chi_0, 0+)\\
      &\qquad= \Psi^{\hat x_0}(u_0, \chi_0, R)-\Psi^{0}(u_0, \chi_0, R)\\
      &\qquad \rightarrow 0\quad \text{as } R\rightarrow \infty,
    \end{split}
  \end{equation}
  where the last line is due to the fact that $u_0$ is homogeneous,
  more precisely, since $u_0$ is homogeneous of degree $s$, then one
  has
$$
(u_0)_{R}^{\hat x_0}(x)=u_0\left(x+({\hat x_0}/{R})\right)\rightarrow
u_0(x)\quad\text{as } R\rightarrow \infty.
$$
Therefore, by \eqref{eq:dimension} and $x\cdot \nabla u_0=su_0$ we
have $\nabla u_0\cdot \hat x_0=0$ in $\R^n$. This implies that $u_0$
is constant in $\hat x_0/|\hat x_0|$ direction. An induction argument
gives that $u_0(x)\equiv u_0(x_n)$. Then from $u_0(0)=0$ we obtain
that $u_0=0$ on all of $\R^{n-1}$, which contradicts to the
nondegeneracy assumption \eqref{eq:nondegenerate}.

The last statement in the proposition follows immediately from the
observation that $\chi=M$ in $\{u(\cdot,0)>0\}$.
\end{proof}

Our last proposition says that among all nonzero blowups at free
boundary points, the one with a flat free boundary as in
Proposition~\ref{prop:basicex} has the smallest Weiss energy.

\begin{proposition}\label{prop:classify2}
  Let $(u,\chi)$ be a limit solution pair. Let $x_0\in \mathcal{F}_u$
  and assume that $u$ is nondegenerate at $x_0$ in the sense of
  \eqref{eq:nondegenerate}. Then
  $\Psi^{x_0}(u,\chi, 0+)\geq 2M|B'_1|$. Moreover,
  $\Psi^{x_0}(u,\chi,0+)=2M|B'_1|$ iff any blow-up limit $u_0$ at
  $x_0$ is $u_0(x)=\alpha2^{-s}(\sqrt{x_1^2+x_n^2}+x_1)^{s}$ up to a
  rotation, with $\alpha=\sqrt{2M/c_0(s)}$.
\end{proposition}

This result is the analogue of Theorem~9.4 in \cite{AP}, with a
similar proof.

\begin{proof}
  Suppose there is $x_0\in \mathcal{F}_u$ such that
  $\Psi^{x_0}(u,\chi,x_0)<2M|B'_1|$. Let $(u_0,\chi_0)$ be a blow-up
  limit at $x_0$ along a sequence
  $(u_{\epsilon_j,\lambda_j}^{x_0},
  \chi_{\epsilon_j,\lambda_j}^{x_0})$
  with $\epsilon_j/\lambda_j^s\rightarrow 0$ (by
  Lemma~\ref{lem:compactnesslemma2}). Then by
  Corollary~\ref{cor:homoblowup}(ii)(iii), $u_0$ is homogeneous of
  degree $s$, and
  \begin{equation}\label{eq:const_fre}
    \Psi^0(u_0,\chi_0,r)\equiv \int_{B'_1}4\chi_0=\Psi^{x_0}(u,\chi,0+)<2M|B'_1|.
  \end{equation}
  Since $u$ satisfies the nondegeneracy assumption at $x_0$, then
  arguing as in Theorem~\ref{thm:maintheorem} we have that $u_0$ is
  nontrivial.

  Let $\Lambda:=\{u_0(\cdot,0)=0\}$. Arguing as in \emph{Step 2}(a) of
  Proposition~\ref{prop:basicex} we have $\chi_0(x)=M$ in
  $\R^{n-1}\setminus \Lambda$. Thus \eqref{eq:const_fre} implies that
  \begin{equation}\label{eq:cap}
    |\Lambda \cap B'_1|>|B'_1|/2.
  \end{equation} 
  We write the homogeneous blow-up limit as $u_0(x)=r^sf(\omega)$,
  with $r=|x|$ and $\omega=x/|x|\in \partial B_1$. Since $u_0$ solves
  $\div(|x_n|^{1-2s}\nabla u_0)=0$ in $\R^n\setminus \Lambda$, $u_0>0$
  in $\R^n\setminus \Lambda$ and $u_0=0$ on $\Lambda$, the function
  $f(\omega)$ satisfies
  \begin{align*}
    \omega_n^{2s-1}\nabla _\omega\cdot (\omega_n^{1-2s}\nabla_\omega )f 
    = s(s-n+1)f&\quad\text{in }\partial B_1\setminus \Lambda, \\
    f=0&\quad\text{on } \partial B_1\cap \Lambda,\\
    f>0&\quad\text{on } \partial B_1\setminus \Lambda.
  \end{align*}
  Thus $f$ is the principal Dirichlet eigenfunction for the weighted
  spherical Laplacian on $\partial B_1\setminus \Lambda_\omega$ with
  $\lambda_0=s(s-n+1)$, where
  $\Lambda_\omega:=\Lambda\cap \partial B_1\subset \partial B_1\cap
  \{\omega_n=0\}$.
  From the variational formulation of the principal eigenvalue and
  using the symmetrization we have, among all $\Lambda_\omega$ with
  $|\Lambda_\omega|$ constant, $\lambda$ takes the minimum iff
  $\Lambda_\omega$ is a spherical cap (here and later by spherical cap
  we mean the `thin' spherical cap lying on $\{\omega_n=0\}$, i.e.\
  the classical spherical cap with center on
  $\partial B_1\cap \{\omega_n=0\}$ intersected with
  $\{\omega_n=0\}$). This follows from the fact that the Steiner
  symmetrization in any spherical variable on $\partial B_1$,
  orthogonal to $x_n$ will decrease the principal eigenvalue, see
  Lemma~9.5 in \cite{AP}, by adding the weight of
  $\omega_n^{1-2s}=(\cos\theta_{n-1})^{1-2s}$ (independent of
  $\theta_1$, \ldots, $\theta_{n-2}$) in the energy functional in the
  proof.

  Let $\lambda^*$ denote the minimum eigenvalue associated with the
  spherical cap $\Lambda_{\omega}^*$ with
  $|\Lambda_{\omega}^*|=|\Lambda_\omega|$. We immediately have
  $\lambda^*\leq \lambda_0$. On the other hand, however, by
  \eqref{eq:cap} we have $\lambda^* > \lambda_0$. This is due to the
  fact that $\lambda_0=s(s-n+1)$ is the principal eigenvalue
  associated with the half sphere, which by \eqref{eq:cap} is
  contained in some $\Lambda_{\omega}^*$ after a rotation. Hence we
  arrive at a contradiction.

  Finally, note that the eigenspace associated with half thin-sphere,
  which without of generality we assume to be
  $\partial B_1\cap \{\omega_1\leq 0, \omega_n=0\}$, is generated by
  $u(x)=(\sqrt{x_1^2+x_n^2}+x_1)^{s}$. By the above argument, if
  $\Psi^{x_0}(u,\chi,0+)=2M|B'_1|$, then any blow-up limit $u_0$ is of
  the form $u_0=c(\sqrt{x_1^2+x_n^2}+x_1)^{s}$, $c>0$ after a
  rotation. By Proposition~\ref{prop:basicex},
  $c=2^{-s}\sqrt{2M/c_0(s)}$.
\end{proof}

At the end of the paper we would like to give an alternative proof of
Theorem~\ref{thm:maintheorem}, without relying on
Corollary~\ref{cor:asymptotic2} in Appendix, but rather using
Proposition~\ref{prop:classify} and some ideas from
Proposition~\ref{prop:classify2}.

\begin{proof}[Alternative Proof of Theorem~\ref{thm:maintheorem}]
  We start again by rescaling \eqref{eq:normal}, to obtain that for
  every $R>0$,
  \begin{equation*}
    |\{u_\lambda (\cdot,0)>0\}\cap \{x_1<0\}\cap B'_R|\rightarrow 0\quad\text{as } \lambda\rightarrow 0,
  \end{equation*}
  which implies that $u_0=0$ $\H^{n-1}$-a.e.\ in $\{x_1<0, x_n=0\}$
  and hence $u_0=0$ everywhere on $\{x_1\leq 0, x_n=0\}$, by
  continuity.  Next, using that $\chi=M$ when $u(\cdot,0)>0$, we also
  have
$$
|\{\chi_\lambda<M\}\cap\{x_1>0\}\cap B_R'|\to 0\quad \text{as
}\lambda\to 0,
$$
implying that $\chi_0=M$ $\H^{n-1}$-a.e.\ in $\{x_1>0, x_n=0\}$.
Hence, by Proposition~\ref{prop:classify}, necessarily $u_0>0$ or
$u_0\equiv 0$ in all of $\{x_1>0, x_n=0\}$. Indeed, if
$\mathcal{F}_{u_0}\cap \{x_1>0,x_n=0\}\neq \emptyset$, then there
exists $\hat x_0\in \mathcal{F}_{u_0}\cap \{x_1>0,x_n=0\}$ such that
$\Psi^{\hat x_0}(u_0,\chi_0,0+)=4M|B'_1|$. Since $u_0$ is homogeneous
by Corollary~\ref{cor:homoblowup}(iii), then
$\lambda \hat x_0\in \mathcal{F}_{u_0}$ for each $\lambda>0$ and
$\Psi^{\lambda \hat x_0}(u_0,\chi_0,0+)=4M|B'_1|$. By the upper
semicontinuity of the map $x \mapsto \Psi^x(u_0,\chi_0,0+)$ we
necessarily have $\Psi^0(u_0,\chi_0,0 +)=4M|B'_1|$. However, this
contradicts Proposition~\ref{prop:classify}. Thus, $u_0>0$ in
$\{x_1>0, x_n=0\}$ and hence, $u_0$ satisfies
$\div(|x_n|^{1-2s}\nabla u_0)=0$ in
$\R^n\setminus\{x_1\leq 0, x_n=0\}\}$. Since $u_0$ is also homogeneous
of degree $s$, writing it as $u_0(x)=r^s f_0(\omega)$, with $r=|x|$
and $\omega=x/|x|$, we see that $f_0$ is a nonnegative eigenfunction
of the weighted spherical Laplacian as in
Proposition~\ref{prop:classify2} in
$\partial B_1\setminus\{\omega_1\leq 0, \omega_n=0 \}$. Hence, $f_0$
is a positive multiple of the explicitly given eigenfunction
$2^{-s}\left((\omega_1^2+\omega_n^2)^{1/2}+\omega_1\right)^{s}$ and
hence
$$
u_0(x)=\alpha 2^{-s}\left((x_1^2+x_n^2)^{1/2}+x_1\right)^{s},
$$
for some $\alpha>0$. Then Theorem~\ref{thm:maintheorem} follows
directly by applying Proposition~\ref{prop:basicex}, as in the first
proof of the theorem.
\end{proof}

\appendix

\section{}\label{s:appendix}
In this appendix we prove the asymptotic development for nonnegative
solutions of $\div(|x_n|^{1-2s}\nabla u)=0$ near the `flat' boundary
points. The proof uses ideas similar to those in Lemma~A.1 and
Corollary~A.1 in \cite{CLW}.

Below, we will denote,
\begin{align*}
  \Lambda &:= \{x\in \R^n:x_n=0, x_1\leq 0\}.\\
  P(x)&:=\frac{1}{2^s}\left(\sqrt{x_1^2+x_n^2}+x_1\right)^{s}.
\end{align*}

\begin{lemma}\label{lem:asymptotic}
  Let $u\in C^{0,s}(B_1)$ be nonnegative,
  $u=0$ on $\Lambda\cap B_1$ and satisfy
  $\div(|x_n|^{1-2s}\nabla u)\leq 0$ in $B_1\setminus\Lambda$. Then
  $u$ has an asymptotic development at the origin
  \begin{equation*}
    u(x)=\alpha P(x) + o (|x|^{s})
  \end{equation*}
  with a constant $\alpha \geq 0$.
\end{lemma}

\begin{proof}
  Let
$$\epsilon(r):=\sup\{\epsilon: u(x)\geq \epsilon P(x) \text{ in }
B_{r}\},\quad 0<r<1.$$
Then $\epsilon(r)$ is a nonincreasing function of $r$, and moreover it
is bounded above by the $C^{0,s}$ norm of $u$. Let
$\alpha:=\lim_{r\to 0+}\epsilon(r)$. From this definition of $\alpha$,
we immediately have
\begin{equation}\label{eq:oneside}
  u(x)\geq \alpha P(x)+ o(|x|^{s})\quad\text{in }  B_1.
\end{equation}

\begin{claim}
  We have $u(x)=\alpha P(x) +o(|x|^{s})$.
\end{claim}
We argue by contradiction. Assume that there are some $\delta_0>0$ and
a sequence $x_k\in B_1$ with $r_k:=|x_k|\rightarrow 0$ such that
\begin{equation}\label{eq:otherside}
  u(x_k)-\alpha P(x_k) \geq \delta_0 r_k^{s}.
\end{equation}
Consider the rescalings
$$
u_k(x):=\frac{u(r_kx)}{r_k^{s}}\quad\text{for } x\in B_{1/r_k},\quad
\overline{x}_k:=\frac{1}{r_k}x_k\in \partial B_1.
$$
Since $u\in C^{0,s}(B_1)$, then there exists a subsequence which we
still denote by $u_k$ and $v\in C^{0,s}_{\loc}(\R^n)$ such that
$u_k\rightarrow v$ uniformly on compact subsets in $\R^n$. We can also
assume $\overline{x}_k\rightarrow \overline{x}\in \partial B_1$. From
\eqref{eq:oneside} and \eqref{eq:otherside} we have
$$v-\alpha P \geq 0 \text{ in } \overline{B_1}, \quad v(\overline{x})-\alpha P(\overline{x})\geq \delta_0. $$
By the uniform convergence of $u_k$ to $v$ and the H\"older regularity
of $v$, there exists $\eta>0$ such that
$\overline{B_\eta (\overline{x})}\cap \Lambda =\emptyset$ and
$$v-\alpha P\geq \frac{\delta_0}{2}, \quad u_k-\alpha P\geq
\frac{\delta_0}{2}\quad\text{for } k>k_0\text{ large enough, on }
B_\eta (\overline{x}).$$
Now let $w$ be a solution to $\div{(|x_n|^{1-2s}\nabla w)}=0$ in
$B_1\setminus \Lambda$ with smooth boundary data, such that
\begin{align*}
  w=0 &\quad\text{on } \partial(B_1\setminus \Lambda)\setminus B_{\eta/2}(\overline{x}),\\
  w= \frac{\delta_0}{4}&\quad\text{on } \partial (B_1\setminus \Lambda)\cap B_{\eta/4}(\overline{x}),\\
  0\leq w\leq \frac{\delta_0}{4}&\quad\text{on } \partial (B_1\setminus \Lambda)\cap B_{\eta/2}(\overline{x}).
\end{align*}
By the maximum principle, $w$ is nonnegative in
$B_1\setminus \Lambda $. By the boundary Harnack principle (see
\cite{CaSalSil}), there exist small $\mu , \gamma>0$ which depend on
$\delta_0$ and $\epsilon $ such that
\begin{equation*}
  w(x)\geq \mu P(x)\quad\text{on } B_\gamma.
\end{equation*}
Now, each $u_k$ satisfies $\div(|x_n|^{1-2s}\nabla u_k)\leq 0$ in
$B_1\setminus\Lambda$. Let $w_k$ be the solution to
$\div(|x_n|^{1-2s}\nabla w_k)=0$ in $B_1\setminus\Lambda$ and
$w_k=\min(0,u_k-\alpha P)+w$ on $\partial (B_1\setminus\Lambda)$. By
the comparison principle, $u_k-\alpha P\geq w_k$ in
$B_1\setminus\Lambda$. Moreover, by \eqref{eq:oneside}, we see
$w_k\rightarrow w$ uniformly on $\partial(B_1\setminus\Lambda)$ and
hence, by the maximum principle, also on $B_1$. By the boundary
Harnack principle, we can assume therefore that
$$
w(x)-w_k(x)\leq \frac{\mu}2P(x),\quad\text{in }B_\gamma,
$$
for large $k$, which then gives
\begin{equation*}
  u_k(x)-\alpha P(x)\geq w_k(x)+\mu P(x)-w(x)\geq
  \frac{\mu}{2}P(x)\quad\text{in }B_\gamma.
\end{equation*}
Scaling back, we therefore have
$$
u(x)\geq \left(\alpha+\frac{\mu}2\right)P(x)\quad\text{in }B_{\gamma
  r_k},
$$
implying that $\epsilon(\gamma r_k)\geq \alpha+\frac{\mu}2$, which in
turn leads to the absurd $\alpha\geq \alpha+\frac{\mu}2$. This proves
the lemma.
\end{proof}

\begin{corollary}\label{cor:asymptotic2}
  Let $u\in C^{0,s}(B_1)$ be nonnegative,
  $u=0$ on $\Lambda\cap B_1$ and satisfy
  $\div(|x_n|^{1-2s}\nabla u)= 0$ in $\{u>0\}$. Then $u$ has an
  asymptotic development at the origin
  \begin{equation*}
    u(x)=\alpha P(x) + o (|x|^{s})
  \end{equation*}
  with a constant $\alpha \geq 0$.
\end{corollary}
\begin{proof} Note that from the conditions above
  $\div(|x_n|^{1-s}\nabla u)\geq 0$ in $B_1\setminus\Lambda$.  Then,
  we claim that there exists $C\geq0$ such that
$$
u\leq CP(x)\quad\text{in }B_{1/2}.
$$
Indeed, if $w$ is a solution of the Dirichlet problem
$\div(|x_n|^{1-2s}\nabla w)=0$ in $B_{3/4}\setminus\Lambda$, $w=u$ on
$\partial (B_{3/4}\setminus \Lambda)$, then by the boundary Harnack
principle
$$
w\leq C P(x)\quad\text{in }B_{1/2}
$$ 
and our claim follows from the comparison $u\leq w$ in $B_{3/4}$. Now,
$$
U(x)=C P(x)-u(x)
$$
will satisfy the conditions of Lemma~\ref{lem:asymptotic} (in
$B_{1/2}$ instead $B_1$) and the asymptotic development of $u$ will
follow from that of $U$.
\end{proof}

\bibliographystyle{plain} \bibliography{biblio}

\begin{bibdiv}
\begin{biblist}

\bib{AC}{article}{
      author={Alt, H.~W.},
      author={Caffarelli, L.~A.},
       title={Existence and regularity for a minimum problem with free
  boundary},
        date={1981},
        ISSN={0075-4102},
     journal={J. Reine Angew. Math.},
      volume={325},
       pages={105\ndash 144},
      review={\MR{618549 (83a:49011)}},
}

\bib{Allen}{article}{
      author={Allen, Mark},
       title={Separation of a lower dimensional free boundary in a two-phase
  problem},
        date={2012},
        ISSN={1073-2780},
     journal={Math. Res. Lett.},
      volume={19},
      number={5},
       pages={1055\ndash 1074},
         url={http://dx.doi.org/10.4310/MRL.2012.v19.n5.a8},
      review={\MR{3039830}},
}

\bib{ALP}{article}{
      author={Allen, Mark},
      author={Lindgren, Erik},
      author={Petrosyan, Arshak},
       title={The two-phase fractional obstacle problem},
        date={2015},
        ISSN={0036-1410},
     journal={SIAM J. Math. Anal.},
      volume={47},
      number={3},
       pages={1879\ndash 1905},
         url={http://dx.doi.org/10.1137/140974195},
      review={\MR{3348118}},
}

\bib{AP}{article}{
      author={Allen, Mark},
      author={Petrosyan, Arshak},
       title={A two-phase problem with a lower-dimensional free boundary},
        date={2012},
        ISSN={1463-9963},
     journal={Interfaces Free Bound.},
      volume={14},
      number={3},
       pages={307\ndash 342},
         url={http://dx.doi.org/10.4171/IFB/283},
      review={\MR{2995409}},
}

\bib{BCN}{incollection}{
      author={Berestycki, H.},
      author={Caffarelli, L.~A.},
      author={Nirenberg, L.},
       title={Uniform estimates for regularization of free boundary problems},
        date={1990},
   booktitle={Analysis and partial differential equations},
      series={Lecture Notes in Pure and Appl. Math.},
      volume={122},
   publisher={Dekker, New York},
       pages={567\ndash 619},
      review={\MR{1044809 (91b:35112)}},
}

\bib{CLW2}{article}{
      author={Caffarelli, L.~A.},
      author={Lederman, C.},
      author={Wolanski, N.},
       title={Pointwise and viscosity solutions for the limit of a two phase
  parabolic singular perturbation problem},
        date={1997},
        ISSN={0022-2518},
     journal={Indiana Univ. Math. J.},
      volume={46},
      number={3},
       pages={719\ndash 740},
         url={http://dx.doi.org/10.1512/iumj.1997.46.1413},
      review={\MR{1488334 (99c:35116)}},
}

\bib{CLW}{article}{
      author={Caffarelli, L.~A.},
      author={Lederman, C.},
      author={Wolanski, N.},
       title={Uniform estimates and limits for a two phase parabolic singular
  perturbation problem},
        date={1997},
        ISSN={0022-2518},
     journal={Indiana Univ. Math. J.},
      volume={46},
      number={2},
       pages={453\ndash 489},
         url={http://dx.doi.org/10.1512/iumj.1997.46.1470},
      review={\MR{1481599 (98i:35099)}},
}

\bib{CRS}{article}{
      author={Caffarelli, Luis~A.},
      author={Roquejoffre, Jean-Michel},
      author={Sire, Yannick},
       title={Variational problems for free boundaries for the fractional
  {L}aplacian},
        date={2010},
        ISSN={1435-9855},
     journal={J. Eur. Math. Soc. (JEMS)},
      volume={12},
      number={5},
       pages={1151\ndash 1179},
         url={http://dx.doi.org/10.4171/JEMS/226},
      review={\MR{2677613 (2011f:49024)}},
}

\bib{caffaSil}{article}{
      author={Caffarelli, Luis},
      author={Silvestre, Luis},
       title={An extension problem related to the fractional {L}aplacian},
        date={2007},
        ISSN={0360-5302},
     journal={Comm. Partial Differential Equations},
      volume={32},
      number={7-9},
       pages={1245\ndash 1260},
         url={http://dx.doi.org/10.1080/03605300600987306},
      review={\MR{2354493 (2009k:35096)}},
}

\bib{CS}{article}{
      author={Cabr{\'e}, Xavier},
      author={Sire, Yannick},
       title={Nonlinear equations for fractional {L}aplacians, {I}:
  {R}egularity, maximum principles, and {H}amiltonian estimates},
        date={2014},
        ISSN={0294-1449},
     journal={Ann. Inst. H. Poincar\'e Anal. Non Lin\'eaire},
      volume={31},
      number={1},
       pages={23\ndash 53},
         url={http://dx.doi.org/10.1016/j.anihpc.2013.02.001},
      review={\MR{3165278}},
}

\bib{CaSalSil}{article}{
      author={Caffarelli, Luis~A.},
      author={Salsa, Sandro},
      author={Silvestre, Luis},
       title={Regularity estimates for the solution and the free boundary of
  the obstacle problem for the fractional {L}aplacian},
        date={2008},
        ISSN={0020-9910},
     journal={Invent. Math.},
      volume={171},
      number={2},
       pages={425\ndash 461},
         url={http://dx.doi.org/10.1007/s00222-007-0086-6},
      review={\MR{2367025 (2009g:35347)}},
}

\bib{CV}{article}{
      author={Caffarelli, Luis~A.},
      author={V{\'a}zquez, Juan~L.},
       title={A free-boundary problem for the heat equation arising in flame
  propagation},
        date={1995},
        ISSN={0002-9947},
     journal={Trans. Amer. Math. Soc.},
      volume={347},
      number={2},
       pages={411\ndash 441},
         url={http://dx.doi.org/10.2307/2154895},
      review={\MR{1260199 (95e:35097)}},
}

\bib{DPS}{article}{
      author={Danielli, D.},
      author={Petrosyan, A.},
      author={Shahgholian, H.},
       title={A singular perturbation problem for the {$p$}-{L}aplace
  operator},
        date={2003},
        ISSN={0022-2518},
     journal={Indiana Univ. Math. J.},
      volume={52},
      number={2},
       pages={457\ndash 476},
         url={http://dx.doi.org/10.1512/iumj.2003.52.2163},
      review={\MR{1976085 (2005d:35066)}},
}

\bib{DR}{article}{
      author={De~Silva, D.},
      author={Roquejoffre, J.~M.},
       title={Regularity in a one-phase free boundary problem for the
  fractional {L}aplacian},
        date={2012},
        ISSN={0294-1449},
     journal={Ann. Inst. H. Poincar\'e Anal. Non Lin\'eaire},
      volume={29},
      number={3},
       pages={335\ndash 367},
         url={http://dx.doi.org/10.1016/j.anihpc.2011.11.003},
      review={\MR{2926238}},
}

\bib{DS1}{article}{
      author={De~Silva, D.},
      author={Savin, O.},
       title={{$C\sp {2,\alpha}$} regularity of flat free boundaries for the
  thin one-phase problem},
        date={2012},
        ISSN={0022-0396},
     journal={J. Differential Equations},
      volume={253},
      number={8},
       pages={2420\ndash 2459},
         url={http://dx.doi.org/10.1016/j.jde.2012.06.021},
      review={\MR{2950457}},
}

\bib{DS2}{unpublished}{
      author={De~Silva, Daniela},
      author={Savin, Ovidiu},
       title={{$C^\infty$} regularity of certain thin free boundaries},
        date={2014},
        note={arXiv:1402.1098},
}

\bib{DSS}{article}{
      author={De~Silva, D.},
      author={Savin, O.},
      author={Sire, Y.},
       title={A one-phase problem for the fractional {L}aplacian: regularity of
  flat free boundaries},
        date={2014},
        ISSN={2304-7909},
     journal={Bull. Inst. Math. Acad. Sin. (N.S.)},
      volume={9},
      number={1},
       pages={111\ndash 145},
      review={\MR{3234971}},
}

\bib{FKS}{article}{
      author={Fabes, Eugene~B.},
      author={Kenig, Carlos~E.},
      author={Serapioni, Raul~P.},
       title={The local regularity of solutions of degenerate elliptic
  equations},
        date={1982},
        ISSN={0360-5302},
     journal={Comm. Partial Differential Equations},
      volume={7},
      number={1},
       pages={77\ndash 116},
      review={\MR{MR643158 (84i:35070)}},
}

\bib{landkof}{book}{
      author={Landkof, N.~S.},
       title={Foundations of modern potential theory},
   publisher={Springer-Verlag},
     address={New York},
        date={1972},
        note={Translated from the Russian by A. P. Doohovskoy, Die Grundlehren
  der mathematischen Wissenschaften, Band 180},
      review={\MR{MR0350027 (50 \#2520)}},
}

\bib{Vazquez-zako}{incollection}{
      author={Vazquez, J.~L.},
       title={The free boundary problem for the heat equation with fixed
  gradient condition},
        date={1996},
   booktitle={Free boundary problems, theory and applications ({Z}akopane,
  1995)},
      series={Pitman Res. Notes Math. Ser.},
      volume={363},
   publisher={Longman, Harlow},
       pages={277\ndash 302},
      review={\MR{1462990 (98h:35246)}},
}

\bib{Weiss}{article}{
      author={Weiss, G.~S.},
       title={A singular limit arising in combustion theory: fine properties of
  the free boundary},
        date={2003},
        ISSN={0944-2669},
     journal={Calc. Var. Partial Differential Equations},
      volume={17},
      number={3},
       pages={311\ndash 340},
      review={\MR{1989835 (2004h:35229)}},
}

\bib{Weiss2}{article}{
      author={Weiss, Georg~Sebastian},
       title={Partial regularity for a minimum problem with free boundary},
        date={1999},
        ISSN={1050-6926},
     journal={J. Geom. Anal.},
      volume={9},
      number={2},
       pages={317\ndash 326},
         url={http://dx.doi.org/10.1007/BF02921941},
      review={\MR{1759450 (2001b:49053)}},
}

\bib{ZFK}{article}{
      author={Zel'dovich, Ya.~B.},
      author={Frank-Kamenetski\u{\i}, D.~A.},
       title={The theory of thermal propagation of flames},
        date={1938},
     journal={Zh.\ Fiz.\ Khim.},
      volume={12},
       pages={100\ndash 105},
        note={(in Russian; English translation in ``Collected Works of Ya. B.
  Zeldovich,'' vol.~1, Princeton Univ. Press 1992.},
}

\end{biblist}
\end{bibdiv}

\end{document}